\documentclass[11pt]{article}

\usepackage{amsthm}
\usepackage{authblk}

\usepackage{graphicx}
\usepackage{amsmath}
\usepackage{amssymb}
\usepackage{enumitem}
\usepackage{hyperref}
\usepackage{nameref}

\usepackage[
  left=3cm,
  right=3cm,
  top=3cm,
  bottom=3cm
]{geometry}

\usepackage{amsthm}
\usepackage{authblk}

\newtheorem{theorem}{Theorem}[section]
\newtheorem{proposition}{Proposition}[section]
\newtheorem{lemma}{Lemma}[section]
\newtheorem{corollary}{Corollary}[section]

\theoremstyle{definition}
\newtheorem{definition}{Definition}[section]
\newtheorem{example}{Example}[section]
\newtheorem{remark}{Remark}[section]

\makeatletter
\def\namedlabel#1#2{\begingroup
    #2%
    \def\@currentlabel{#2}%
    \phantomsection\label{#1}\endgroup}
\makeatother

\newenvironment{claim}[1]{\par\medskip\noindent\emph{Claim #1:}\space}{\par\medskip}
\newenvironment{claimproof}[1]{\par\noindent\emph{Proof of Claim #1:}\space}{\hfill $\square$\par\medskip}

\renewcommand{\H}{\mathcal{H}}
\newcommand{\tto}{\rightrightarrows}
\newcommand{\dd}{\, \mathrm{d}}
\newcommand{\gph}{\operatorname{gph}}

\usepackage{authblk}

\begin{document}

\title{A Filippov theorem for Volterra sweeping processes with one-sided Lipschitz perturbation}

\author[1]{Abderrahim Jourani\thanks{\texttt{abderrahim.jourani@u-bourgogne.fr}}}
\author[2,3]{Diana Narv\'aez\thanks{\texttt{diana.narvaez@postdoc.uoh.cl}}}
\author[2]{Emilio Vilches\thanks{\texttt{emilio.vilches@uoh.cl}}}

\affil[1]{Institut de Math\'ematiques de Bourgogne, UMR 5584, CNRS,
Universit\'e Bourgogne Europe, Dijon, France}

\affil[2]{Universidad de O’Higgins,
Instituto de Ciencias de la Ingeniería, Rancagua, Chile}

\affil[3]{Universidad de Chile,
Centro de Modelamiento Matemático, Santiago, Chile}

\date{\today}

\maketitle

\begin{abstract}
We prove a Filippov-type stability theorem for integro-differential sweeping processes of Volterra type with an outer multivalued perturbation, in a separable Hilbert space, under the assumptions that the moving sets are uniformly prox-regular and Lipschitz continuous with respect to the Hausdorff distance and that the perturbation is one-sided Lipschitz. Given an absolutely continuous solution of the system perturbed both in the state argument of the multivalued term and by an outer integrable term, we show the existence of a solution of the original Volterra sweeping process and estimate the distance between the two trajectories explicitly in terms of the data of the problem. The estimate becomes sharper when only the outer perturbation is present, and it recovers the Lipschitz dependence on the initial condition of Filippov's classical theorem. The proof combines a reduction of the constrained dynamics to an unconstrained differential inclusion, measurable selection arguments, and an enhanced version of Gr\"onwall's inequality established in this work. As an application, we obtain the Lipschitz dependence of the attainable set on the initial set with respect to the Hausdorff distance, together with a one-sided estimate quantifying the effect of the perturbations. We also work out a spatially distributed fishery model with ecological memory whose harvesting rule, triggered by the aggregate biomass, is one-sided Lipschitz but is not Lipschitz continuous with respect to the Hausdorff distance, so that the classical framework does not apply. For this model, our estimates produce a stability constant governed only by the biological data.
\end{abstract}


\section{Introduction}
Let $I=[0,T]$ a nonempty closed interval and  $\H$ a separable Hilbert space. This paper is concerned with the stability of an evolution problem in which three different nonsmooth mechanisms act simultaneously:
\begin{equation}\label{Sweeping-Dif1}\small{
\left\{
\begin{aligned}
\dot{x}(t)&\in -N(C(t);x(t))+f_1(t,x(t))+\int_{0}^t f_2(t,s,x(s))\, ds +F(t,x(t)) & \textrm{a.e. }  t\in I,\\
x(0)&=x_0\in \mathcal{K}\subset C(0),
\end{aligned}
\right.}
\end{equation}
The unbounded normal-cone term $-N(C(t);x(t))$ confines the state to a moving set $C(t)$ that is merely uniformly $\rho$-prox-regular, and generates the impulsive velocities typical of unilateral dynamics; the integral term makes the velocity at time $t$ depend on the whole history of the trajectory through a kernel of \emph{Volterra type}, i.e., one depending on the current time $t$ as well as on the past time $s$; and the outer multivalued map $F$ carries the controls, the uncertainties, or the discontinuous feedback laws of the model. The question we answer is the sweeping-process analogue of the one settled by Filippov for differential inclusions: \emph{if a trajectory satisfies \eqref{Sweeping-Dif1} only up to a measurable defect, is there an exact trajectory nearby, and can the distance between the two be bounded explicitly in terms of the size of the defect and of the data of the problem alone?}

For differential inclusions, the affirmative answer is Filippov's celebrated theorem \cite{Filippov-1967_Classical}, which underlies most of the quantitative theory of reachable sets and of relaxation. Its driving hypothesis is a Lipschitz condition, in the state variable, on the right-hand side. That hypothesis is, however, badly suited to systems in which the multivalued term encodes a switching or threshold law, and this motivated the introduction by Donchev of the \emph{one-sided Lipschitz} (OSL) condition \cite{Donchev_T-1991-Functional,Donchev_T-2002}, which constrains only the inner product $\langle x-y,v-w\rangle$ and therefore tolerates arbitrarily large, even infinite, Hausdorff moduli, provided that the offending variation points in a dissipative direction. An important distinction with respect to the classical Filippov theorem is that, under the Lipschitz assumption, it provides estimates for both the trajectories and their velocities \cite{Filippov-1967_Classical}, while the Filippov-type result of Donchev and Farkhi \cite{Dontchev_Farkhi-1998}, obtained under the weaker OSL condition, provides a stability estimate for the trajectories but not for their velocities. Donchev and Farkhi showed that stability and Euler approximation persist under OSL \cite{Dontchev_Farkhi-1998}. They later revisited the Filippov--Pli\'{s} theorem under one-sided Lipschitz, one-sided Kamke, or continuity-type conditions \cite{Donchev_2009_Th_Apli}. In parallel, Donchev, R\'ios and Wolenski characterized strong invariance under OSL assumptions in finite dimensions \cite{Dontchev_T-Rios-Wolenski-2005}, in Hilbert spaces and for discontinuous dynamics \cite{Donchev_2006_Hilbert}, and for perturbed dissipative systems \cite{Dontchev_T-Rios-Wolenski-2004_C}, thereby extending the characterization of \cite{RW2003}; ramifications in Hamilton--Jacobi theory, concerning approximate semi-solutions and the minimal time function, appear in \cite{Dontchev_T-Rios-Wolenski-2007}. It should be stressed that the OSL condition used here (Definition~\ref{def:OSL}) is the weakest of the conditions bearing that name: it requires that for \emph{each} $v\in F(t,x)$ \emph{some} $w\in F(t,y)$ realize the inequality,  as in \cite{Dontchev_Farkhi-1998,Donchev_T-Farkhi_E-1999-Approximations} whereas the \emph{uniform} one-sided Lipschitz condition of \cite{Filippov-1960,Filippov-1988,Dontchev_A-Lempio_F-1992,Kastner-1991}. Under OSL conditions, the successive approximations used in Filippov’s original proof may fail to converge, as pointed out in \cite{Donchev_2009_Th_Apli} and illustrated by the classical counterexample of Lakshmikantham and Leela \cite{Lakshmikantham-1981}. Instead, our argument uses existence results for the associated differential inclusions under suitable compactness assumptions and provides estimates for the trajectories, rather than for their velocities. 

On the constrained side, the sweeping process was introduced by Moreau in the early seventies \cite{MO1,MO2,MO4} to model quasi-static elastoplasticity, and has become a standard tool in unilateral mechanics, in nonsmooth dynamics and in their numerical treatment \cite{Moreau1999}. We refer to \cite{MR755330,MR1189795} for the general theory of differential inclusions and to \cite{Vilches2024,JV-alpha,JV-regular,JV-Galerkin,MR3956966,MR3813128,MR4421900} for existence and regularization results covering several perturbed and state-dependent variants. Integro-differential sweeping processes, in which memory is added to the dynamics, have been intensively studied over the last few years. First, Colombo and Kozaily \cite{MR4099068} established existence and uniqueness for an integral perturbation of Moreau's process. Then Bouach, Haddad and Thibault \cite{bouach2021nonconvex} treated the nonconvex case with applications. The corresponding optimal control and discretization problems being addressed in \cite{bouach2021optimal,Haddad2022}. Then,  Gaouir, Haddad and Thibault investigated the prox-regular setting \cite{HaddadGaouirThibault2024} and its \emph{Lipschitz} multivalued perturbations \cite{HaddadGaouirThibault2025}. The well-posedness of the Volterra-type variant considered here was obtained in \cite{Vilches2024}, and a Galerkin-like method for integro-differential inclusions, with applications to Volterra sweeping processes, was developed in \cite{Pedro-Manuel-Emilio-2024}. Integral dynamics of a related nature also arises in the theory of sticky particles \cite{MR3039208}.

This body of work leaves a precise gap. For integro-differential sweeping processes carrying a multivalued perturbation,  the available framework in \cite{HaddadGaouirThibault2025} requires $F(t,\cdot)$ to be Lipschitz continuous with respect to the Hausdorff distance, and the reason is structural: Hausdorff-Lipschitz continuity allows one to choose, at each step of the successive-approximation scheme, a measurable selection whose distance from the preceding one is controlled by the distance between the corresponding states. Under the OSL condition alone, no analogous norm estimate between matching selections is available. In the OSL setting considered here, the presence of the Volterra term leads to quadratic estimates involving nonlocal terms that are not directly covered by the classical Grönwall inequality. The present paper closes this gap. Its message can be summarized as follows: \emph{the entire Filippov theory of Volterra sweeping processes survives the replacement of Hausdorff-Lipschitz continuity by the one-sided Lipschitz condition, with fully explicit constants, and the gain is genuine rather than cosmetic.} OSL assumption (see \ref{HOSL} below) is strictly weaker than Hausdorff-Lipschitz continuity (Remark~\ref{rem:OSL-normalization}(ii)), and Section~\ref{subsec:fishery} exhibits a natural infinite-dimensional model, namely a spatially distributed fishery with ecological memory and a threshold harvest-control rule, in which $F$ is OSL with constant \emph{zero} while failing to be Hausdorff-Lipschitz on \emph{every} ball, so that the framework of \cite{HaddadGaouirThibault2025} does not apply at all, whereas ours yields a stability constant governed exclusively by the biological data.

Our main result, Theorem~\ref{Filippov_T}, may be stated informally as follows. Let $y(\cdot)$ be an absolutely continuous solution of the doubly perturbed inclusion obtained from \eqref{Sweeping-Dif1} by perturbing the state argument of $F$ by an essentially bounded \emph{inner} perturbation $h$ and by adding an \emph{outer} perturbation $g(t)\mathbb{B}$ with $g$ integrable. Then, for every  compact set $\mathcal{K}\subset C(0)$ and every $x_{0}\in \operatorname{Proj}_{\mathcal{K}}(y(0))$, there exists a genuine solution $x(\cdot)$ of \eqref{Sweeping-Dif1} issued from $x_{0}$ such that $\Vert x(t)-y(t)\Vert$ is bounded, for every $t\in I$, by an expression built only from $\operatorname{dist}(y(0),\mathcal{K})$, from $\Vert h\Vert_{\infty}$ and $\Vert g\Vert_{L^1}$, and from the data through the one-sided Lipschitz modulus $L_{F}^{\overline{R}}$ of $F$, the local Lipschitz moduli $l^{1}_{\overline{R}}$ and $l^{2}_{\overline{R}}$ of $f_1$ and $f_2$, the Lipschitz constant $L_{C}$ of the moving set and its prox-regularity constant $\rho$. No unspecified constant enters the estimate. When the inner perturbation is switched off, the bound sharpens from a quadratic to a linear regime, and in the unperturbed case $\mathcal{K}=\{x_0\}$, $h\equiv 0$, $g\equiv 0$ it collapses to
\begin{align*}
\Vert x(t)-y(t)\Vert &\leq \Vert x_{0}-y_{0}\Vert \exp\left(\int_0^t \eta(s)\, ds\right),\\
\textrm{ where } \qquad \eta(s)&:=L_{F}^{\overline{R}}(s)+l^{1}_{\overline{R}}(s)+\frac{L_{C}+\nu(s,0,0)}{\rho}+s\,l^{2}_{\overline{R}}(s),
\end{align*}
which is exactly the Lipschitz dependence on the initial condition of Filippov's classical theorem. The exponent is transparent: each datum contributes its own modulus, the memory kernel contributes the factor $t\,l^{2}_{\overline{R}}(t)$ produced by the Volterra structure, and the nonconvexity of the moving set contributes the term $(L_{C}+\nu)/\rho$, which is proportional to the a priori velocity bound and vanishes identically when the sets $C(t)$ are convex, i.e., when $\rho=+\infty$. To the best of our knowledge, this is the first Filippov-type theorem for integro-differential sweeping processes of Volterra type under a one-sided Lipschitz condition on the multivalued perturbation.

The contributions of the paper are the following.\\
(C1)  \textit{A Filippov-type stability theorem under OSL} (Theorem~\ref{Filippov_T}). The Lipschitz continuity of $F(t,\cdot)$ for the Hausdorff distance is relaxed to \ref{HOSL}. The result is moreover stated with two simultaneous perturbations, inner and outer, with set-valued initial data, and in two regimes: a quadratic estimate in the general case and a sharper linear one when $h\equiv 0$, the former reducing exactly to the square of the latter when both perturbations vanish. Existence of the comparison trajectory is part of the conclusion, and follows from either of two alternative compactness hypotheses, one on the moving set \ref{HC3} and one on the perturbation \ref{H4F}. Neither is redundant: in the fishery model of Section~\ref{subsec:fishery} the former fails in infinite dimensions, while the latter holds.\\
(C2) \textit{An enhanced Gr\"onwall inequality} (Proposition~\ref{Gronwall}), covering differential inequalities in which a square-root nonlinearity and Volterra memory occur together, including the product of a local term $\sqrt{\beta(t)}$ with a nonlocal one $\int_{0}^{t}\sqrt{\beta}(s)ds$. This is the analytic core of the paper, and its form is dictated by the problem, as explained below.\\
(C3) \textit{Applications} (Section~\ref{section-appl}). The attainable set $A(t,\mathcal{K})$ depends on $\mathcal{K}$ in a Lipschitz way for the Hausdorff distance, with a modulus that is uniform over families of compact subsets of $C(0)$ contained in a fixed ball (Corollary~\ref{cor:attainable}); and a one-sided estimate quantifies the drift of the attainable set of the perturbed dynamics (Remark~\ref{rem:perturbed-attainable}), the excess being, as we explain there, the only quantity that can legitimately be estimated in this direction. Finally, Proposition~\ref{prop:fishery} works out the fishery model in full, and shows that the resulting stability constant $\exp\bigl(\int_{0}^{t}(q(s)+s\,\kappa(s))\,ds\bigr)$ depends only on the recruitment and mortality rates and on the memory kernel --- and not on the admissible fishing-mortality range, on the management set, or on the threshold rule. It is thus strictly better than the sharpest Hausdorff modulus available even in those degenerate cases where the Lipschitz framework does apply.\\

It may be worth explaining why the passage from the Lipschitz to the one-sided Lipschitz setting is not a routine adaptation. Three obstructions have to be removed. First, the OSL condition bears on inner products, hence it controls the \emph{squared} distance $\varphi=\frac12\Vert x-y\Vert^{2}$, whereas the perturbations $h$ and $g$ enter linearly in $\Vert x-y\Vert$. Differentiating $\varphi$ therefore produces a differential inequality in which $\varphi$, $\sqrt{\varphi}$ and a constant appear at the same time. Second, the Volterra kernel introduces memory terms in a particularly delicate form, multiplying the local factor $\sqrt{\varphi(t)}$ by the nonlocal term $\int_{0}^{t}\sqrt{\varphi}$. Neither the classical Gr\"onwall inequality nor its usual Volterra variants accommodate this combination, which is precisely what Proposition~\ref{Gronwall} is designed to handle. Third,  we construct an auxiliary multifunction $S$ by intersecting $F$, evaluated at a suitably truncated argument because the OSL modulus is only local, with the half-space determined by the OSL inequality relative to the given trajectory $y$. We then verify that $S$ inherits from $F$ its measurability, the weak closedness of its graph, its linear growth, and its measure-of-noncompactness estimate. Applying an existence theorem to $S$ yields a solution of \eqref{Sweeping-Dif1} that is, by construction, OSL-matched to $y$ at almost every time. A fourth, more subtle point is that $x$ and $y$ satisfy inclusions with \emph{different} a priori velocity bounds while being swept by the \emph{same} moving set.   We handle this by applying a reduction result (see Proposition \ref{Main_Result_Red}) with a coefficient that dominates both bounds. This is the role of the free parameter $r_{0}$ and the comparison property $\nu(t,0,0)\leq \nu(\cdot,h,g)$.

The paper is organized as follows. Sections~\ref{Math_Prel} and \ref{preliminary_L} collect the material from nonsmooth analysis and the theory of measurable multifunctions needed later, including a measurability lemma, an enhanced Gr\"onwall inequality, and a discussion of the OSL property together with four classes of examples: parametrized dynamics, subdifferentials of weakly convex functions, monotone perturbations, and normal-cone perturbations over prox-regular sets. Section~\ref{hipo-sol} states the standing assumptions on the data. Section~\ref{Sweeping-sec} establishes the reduction theorem and the corresponding a priori bounds. Section~\ref{Main_result} contains the main result. Section~\ref{section-appl} develops the applications, namely the Hausdorff continuity of attainable sets in Subsection~\ref{A_Opt_Cont} and the fishery model in Subsection~\ref{subsec:fishery}. Section~\ref{Concl} concludes with several perspectives, in particular on strong invariance and minimal time problems, as well as on weak asymptotic stability.

\section{Mathematical Preliminaries}\label{Math_Prel}
From now on, $\H$ stands for a separable Hilbert space, whose inner product is denoted by $\langle\cdot,\cdot\rangle$ and whose norm is denoted by $\Vert \cdot \Vert$. The closed unit ball is denoted by $\mathbb{B}$. The notation $\H_w$ stands for $\H$ equipped with the weak topology, and $x_n \rightharpoonup x$ denotes the weak convergence of a sequence $(x_n)_n$ to $x$. Throughout the paper, $T>0$ is fixed and $I:=[0,T]$. We denote by $L^1\left(I;\H\right)$ the space of $\H$-valued (Bochner) Lebesgue integrable functions defined over $I$.  Moreover, we say that $u$ is absolutely continuous if there exist $f\in L^1\left(I;\H\right)$ and $u_0\in \H$ such that $u(t)=u_0+\int_{0}^t f(s)\, ds$ for all $t\in I$.

The interval $I$ is endowed with the $\sigma$-field $\mathcal{I}$ of its Lebesgue measurable subsets, which is complete with respect to the ($\sigma$-finite) Lebesgue measure, and $\mathcal{B}(\H)$ denotes the Borel $\sigma$-field of $\H$. Measurability on $I$ is always understood with respect to $\mathcal{I}$.

Let $S$ be a nonempty closed subset of $\H$ and let $x\in S$. A vector $h\in \H$ belongs to the \emph{Clarke tangent cone} $T(S;x)$ when for every sequence $(x_n)_n$ in $S$ converging to $x$ and every sequence of positive numbers $(t_n)_n$ converging to $0$, there exists some sequence $(h_n)_n$ in $\H$ converging to $h$ such that $x_n+t_nh_n\in S$ for all $n\in \mathbb{N}$. This cone is closed and convex and its negative polar $N(S;x)$ is the \emph{Clarke normal cone} to $S$ at $x\in S$, that is,
\begin{equation*}
N\left(S;x\right):=\left\{v\in \H\colon \left\langle v,h\right\rangle \leq 0 \quad  \forall h\in T(S;x)\right\}.
\end{equation*}
As usual, $N(S;x)=\emptyset$ if $x\notin S$. In what follows, we write indistinctly $N(S;x)$ or $N_{S}(x)$ for this cone. Through that normal cone, the Clarke subdifferential of a lower semicontinuous extended-real-valued function $f\colon \H\to \mathbb{R}\cup\{+\infty\}$ is defined, at any point $x\in \H$ where $f(x)$ is finite, by
\begin{equation*}
\partial f(x):=\left\{v\in \H\colon (v,-1)\in N\left(\operatorname{epi}f,(x,f(x))\right)\right\},
\end{equation*}
where $\operatorname{epi}f:=\left\{(y,r)\in \H\times \mathbb{R}\colon f(y)\leq r\right\}$ is the epigraph of $f$. When the function $f$ is finite and locally Lipschitzian around $x$, the Clarke subdifferential is characterized (see, e.g., \cite{Clarke1998}) in the following simple and amenable way
\begin{equation}\label{eq:clarke-support}
\partial f(x)=\left\{v\in \H\colon \left\langle v,h\right\rangle \leq f^{\circ}(x;h) \textrm{ for all } h\in \H\right\},
\end{equation}
where
\begin{equation*}
f^{\circ}(x;h):=\limsup_{(t,y)\to (0^+,x)}t^{-1}\left[f(y+th)-f(y)\right]
\end{equation*}
is the \emph{generalized directional derivative} of the locally Lipschitzian function $f$ at $x$ in the direction $h\in \H$.  The function $f^{\circ}(x;\cdot)$ is in fact the support function of $\partial f(x)$, which is a nonempty, convex and weakly compact subset of $\kappa\mathbb{B}$ whenever $f$ is Lipschitz with constant $\kappa$ near $x$. That characterization easily yields that the Clarke subdifferential of any locally Lipschitzian function has the important property of upper semicontinuity from $\H$ into $\H_w$.

For $x\in \H$ and a nonempty set $S\subset \H$, the distance function to the set $S$ at $x\in \H$ is defined by $d_{S}(x):=\inf_{y\in S}\Vert x-y\Vert$. We denote $\operatorname{Proj}_{S}(x)$ the set (possibly empty)
\begin{equation*}
\operatorname{Proj}_{S}(x):=\left\{y\in S\colon d_{S}(x)=\Vert x-y\Vert\right\}.
\end{equation*}
Whenever $\operatorname{Proj}_{S}(x)$ is a singleton, we denote by $\operatorname{proj}_{S}(x)$ its unique element. This is the only situation in which the lower-case notation is used in the paper.
In the sequel, we use indistinctly the notations $d_{S}(\cdot)$ and $d(\cdot,S)$ and, as usual, it will be convenient to write $\partial d(x,S)$ in place of $\partial d\left(\cdot,S\right)(x)$. A vector $v\in \H$
is a proximal normal vector to  $S$ of $\H$ at $x\in S$ if there exist $\sigma\geq 0$ such that
\begin{equation*}
  \langle v,x'-x \rangle \leq \sigma \|x'-x\|^2 \quad\text{for all }x'\in  S.
\end{equation*}
The set of such vectors is the proximal normal cone $N^P(S,x)$ to $S$ at $x$.
The proximal normal cone enjoys a geometrical characterization (see, e.g., \cite{Clarke1998}):
\begin{equation}\label{eq2.5}
N^{P}(S,x)=\{v \in \H:\;\exists \rho>0\text{ s.t. }x\in \operatorname{Proj}_{S}(x+\rho v )\}.
\end{equation}
Moreover, the neighborhood $U$ in the above definition can be dispensed with at the cost of possibly enlarging the constant (see, e.g., \cite{Clarke1998}), which gives the following variational characterization:
$$
N^{P}(S, x) = \left\lbrace v\in \H\colon \, \exists \sigma \geq 0 \textrm{ such that } \langle v, x'-x\rangle \leq \sigma \Vert x' - x\Vert^2 \,   \forall x'\in S\right\rbrace.
$$
A subdifferential is associated with the proximal normal cone through the epigraph of functions. More precisely, for a lower semicontinuous function $f\colon \H \to \mathbb{R}\cup\{+\infty\}$ and $x\in \H$, with $f(x)<+\infty$, the proximal subdifferential of $f$ at $x$ is
the set
 $$
 \partial^{P}f(x) = \{ v\in \H\colon  (v, -1)\in N^{P}(\operatorname{epi} f, (x, f(x)))\}.
 $$
The uniformity of the positive constant $\rho$ in (\ref{eq2.5}) for the unit proximal normal
vectors to $S$ leads to the concept of uniformly prox-regular sets. For a given $\rho \in ]0,+\infty ]$
the closed subset $S$ is $\rho$-uniformly prox-regular (see, e.g., \cite{P_R_T-2000}) if every unit proximal normal vector to $S$
can be realized by a $\rho$-ball, which means that for all $x\in S$ and all
$ v \in N^{P}(S,x)$ one has
\begin{equation}\label{eq:prox-reg-def}
  \left\langle v ,x'-x \right\rangle
  \leq \frac{\left\Vert v \right\Vert}{2\rho }\left\Vert x'-x\right\Vert ^{2} \textrm{ for all }  x'\in S.
\end{equation}
We make the convention $\frac{1}{\rho }=0$ for
$\rho =+\infty .$ Recall that for $\rho =+\infty $ the uniform $\rho$-prox-regularity
of the closed set $S$ is equivalent to the convexity of the set.
The following proposition provides some properties of the proximal and Clarke subdifferentials
of the distance function $d(\cdot,S)$ when the set $S$ is prox-regular.
It also summarizes some important consequences of the
uniform prox-regularity which will be needed in the paper. For the proofs of these
results we refer the reader to the paper \cite{P_R_T-2000}.

\begin{proposition}\label{Prop2.1}
Let $S\subset \H$ be closed and let $\rho\in ]0,\infty]$. Then the following are equivalent:
\begin{enumerate}[label=(\alph*)]
\item The set $S$ is uniformly $\rho$-prox-regular.
\item For any  $x'\in U_{\rho}(S):=\{x\in \H:\;d(x,S)<\rho\}$ the set
$\operatorname{Proj}_{S}(x')$ is a singleton and the mapping $\operatorname{proj}_{S}(\cdot)$ is continuous on $U_{\rho}(S)$.
\item The proximal subdifferential of $d(\cdot ,S)$ coincides with its Clarke subdifferential at all points $x\in U_{\rho}(S)$.
\item The  function $d(\cdot,S)$ is continuously Fr\'echet differentiable
on $U_{\rho}(S)\setminus S$, with 
$\nabla d(\cdot,S)(x)=\dfrac{x-\operatorname{proj}_{S}(x)}{d(x,S)}$ for every $x\in U_{\rho}(S)\setminus S$. 
\item The set-valued mapping  $x\rightrightarrows N^P(S, x)\cap\mathbb{B}$ is $1/\rho$-hypomonotone, that is,   for all $x_i\in S$ and all $v_i\in N^{P}(S,x_i)\cap \mathbb{B}$, $i=1,2$, one has
\begin{equation*}
\left\langle v_{1}-v_{2},x_{1}-x_{2}\right\rangle \geq -\frac{1}{\rho}\left\Vert
x_{1}-x_{2}\right\Vert ^{2}.
\end{equation*}
\end{enumerate}
\end{proposition}
Assertion (d) cannot be extended to the points of $S$: the function $d(\cdot,S)$ fails to be differentiable at any $x\in S$ admitting a nonzero normal vector. What survives at such points, and what is used in Section~\ref{Sweeping-sec}, is the differentiability of the \emph{squared} distance.
\begin{corollary}\label{cor:squared-distance}
Let $S\subset \H$ be closed and uniformly $\rho$-prox-regular for some $\rho\in ]0,+\infty]$. Then $\frac{1}{2}d^{2}(\cdot,S)$ is continuously differentiable on $U_{\rho}(S)$ and 
\begin{equation*}
\nabla\left(\tfrac{1}{2}d^{2}(\cdot,S)\right)(x)=x-\operatorname{proj}_{S}(x)\qquad \textrm{ for all } x\in U_{\rho}(S).
\end{equation*}
\end{corollary}
\begin{proof}
On $U_{\rho}(S)\setminus S$ the identity follows from Proposition~\ref{Prop2.1}(d) and the chain rule. If $x\in S$, then $\operatorname{proj}_{S}(x)=x$ and, for every $u\in \H$, $0\leq \frac12 d^{2}_{S}(x+u)\leq \frac12\Vert u\Vert^{2}=o(\Vert u\Vert)$, so that $\frac12 d^{2}_{S}$ is Fr\'echet differentiable at $x$ with zero gradient, which is again $x-\operatorname{proj}_{S}(x)$. Finally, the gradient is continuous on $U_{\rho}(S)$ by Proposition~\ref{Prop2.1}(b).
\end{proof}
\noindent Let $S$ be a closed subset of $\H$. Then, for every $x\in S$, $\partial^{P}d_{S}(x)=N^{P}(S;x)\cap \mathbb{B}$.  If, in addition, $S$ is uniformly $\rho$-prox-regular for some $\rho\in ]0,+\infty]$, then, for every $x\in S$,
\begin{equation}\label{lem:trunc}
\partial d_{S}(x)=\partial^{P}d_{S}(x)=N^{P}(S;x)\cap \mathbb{B}=N(S;x)\cap \mathbb{B}.
\end{equation}
In particular, $N^{P}(S;x)=N(S;x)$ for every $x\in S$.

\section{Preliminary Lemmas}\label{preliminary_L}

Throughout, $\mathcal I$ denotes the Lebesgue $\sigma$-algebra on $I=[0,T]$, and $\mathcal B(\H)$ denotes the Borel $\sigma$-algebra of $\H$. The following result provides sufficient conditions for the existence of a measurable minimum-norm selection of the map $t\mapsto F(t,u(t))$ (see, e.g., \cite[Lemma~3.1]{Pedro-Manuel-Emilio-2024}).
\begin{lemma}\label{Lemma_exis}
     Let $F\colon I \times \H\rightrightarrows \H$  be a set-valued map with nonempty, closed and convex values satisfying
\begin{itemize}
        \item The set-valued map $t\rightrightarrows \gph F(t,\cdot)$ is measurable, that is, for every (norm) open set $U\subset \H \times \H$, the set $\{ t \in I\colon \gph F(t,\cdot)\cap U\neq \emptyset \}$ belongs to $\mathcal{I}$.
        \item For a.e. $t\in I$, $\gph F(t,\cdot)$ is closed in $\mathcal{H}\times \mathcal{H}_w$.
\end{itemize}
Then, for any $ u\colon I \to \H$ measurable, the map $M(t):= F(t,u(t)) $ is measurable.  Moreover, $M$ admits a measurable selection $v\colon I \to \H$ such that $d\left(0,F(t,u(t) )\right) =\| v(t)\|$ for a.e. $t\in I$.
\end{lemma}
The following result may be obtained from \cite[Theorems~III.9 and~III.30]{Castaing_Valadier-1977}.
\begin{proposition}\label{meassum}
Let  $F_1\colon  I \rightrightarrows \H$ be a ${\cal I}$-measurable set-valued map
with nonempty closed values and $F_2\colon  I \rightrightarrows \H$  be a set-valued map  with nonempty
values whose graph $\gph(F_2)$  is ${\cal I}\otimes {\cal B}(\H)$-measurable. Then, for any ${\cal I}$-measurable
selection $f$ of $F_1+F_2$  there exist ${\cal I}$-measurable selections $f_1$ and $f_2$ of
$F_1$ and $F_2$  respectively, such that $f(t)=f_1(t)+f_2(t)$ for all $t\in I$.
\end{proposition}
\noindent The next lemma provides the measurability of the multifunctions generated by the moving set along a measurable trajectory. It is used repeatedly in Sections~\ref{Sweeping-sec} and~\ref{Main_result} in order to apply Proposition~\ref{meassum}, and it does \emph{not} require the trajectory to be viable in the moving set.
\begin{lemma}\label{lem:meas-subdiff}
Let $u\colon I\to \H$ be measurable, and let  $C\colon I \rightrightarrows \H$ be a set-valued map with nonempty, closed and uniformly $\rho$-prox-regular values, for some $\rho\in ]0,+\infty]$. Assume that there exists $L_{C}\geq 0$ such that
\begin{equation*}
\operatorname{Haus}(C(t),C(s)):=\sup_{z\in \H}|d(z,C(t))-d(z,C(s))|\leq L_{C} |t-s|  \text{ for all } s,t\in I.
\end{equation*}
Then the following assertions hold.
\begin{enumerate}[label=\textnormal{(\roman*)}]
\item The set-valued map $t\rightrightarrows \partial d_{C(t)}(u(t))$ is measurable, with nonempty, convex and weakly compact values contained in $\mathbb{B}$.
\item If, in addition, $u(t)\in C(t)$ for a.e. $t\in I$, then the set-valued map $t\rightrightarrows N_{C(t)}(u(t))$ is measurable, with nonempty closed values.
\end{enumerate}
\end{lemma}
\begin{proof}
(i) For every $t\in I$ the function $d_{C(t)}(\cdot)$ is $1$-Lipschitz on $\H$. Hence $\partial d_{C(t)}(u(t))$ is a nonempty, convex and weakly compact subset of $\mathbb{B}$. Fix $w\in \H$ and set $\Theta_{w}(t):=d^{\circ}_{C(t)}(u(t);w)$. Let $Q$ be a countable dense subset of $\H$. For $z\in Q$ and $s\in \mathbb{Q}$ with $s>0$, put
\begin{align*}
\alpha_{z,s}(t):=\frac{d_{C(t)}(z+sw)-d_{C(t)}(z)}{s},\qquad
E_{z,n}:=\left\{t\in I\colon \Vert z-u(t)\Vert<\tfrac1n\right\}.
\end{align*}
By assumption, the map $t\mapsto d_{C(t)}(y)$ is $L_{C}$-Lipschitz on $I$ for each fixed $y\in \H$, so that every $\alpha_{z,s}$ is continuous. Moreover $E_{z,n}\in \mathcal{I}$, because $u$ is measurable, and $\vert\alpha_{z,s}\vert\leq \Vert w\Vert$, because $d_{C(t)}$ is $1$-Lipschitz. Since, for each fixed $t$, the map $(y,s)\mapsto s^{-1}\left(d_{C(t)}(y+sw)-d_{C(t)}(y)\right)$ is continuous on $\H\times ]0,+\infty[$, the supremum defining
\begin{align*}
\Sigma_{n}(t):=\sup\left\{\frac{d_{C(t)}(y+sw)-d_{C(t)}(y)}{s}\colon \Vert y-u(t)\Vert<\tfrac1n,\ 0<s<\tfrac1n\right\}
\end{align*}
is unchanged if $y$ is restricted to $Q$ and $s$ to $\mathbb{Q}$. Consequently,
\begin{align*}
\Sigma_{n}=\sup_{z\in Q,\ s\in \mathbb{Q}\cap ]0,1/n[}\ \left(\alpha_{z,s}\,\mathbf{1}_{E_{z,n}}-\Vert w\Vert\,\mathbf{1}_{I\setminus E_{z,n}}\right),
\end{align*}
which is a countable supremum of measurable functions, hence measurable (the truncation by $-\Vert w\Vert$ does not affect the supremum, since $\alpha_{z,s}\geq -\Vert w\Vert$ and, by density of $Q$, for every $t$ there is at least one $z\in Q$ with $t\in E_{z,n}$). Therefore $\Theta_{w}=\inf_{n\in \mathbb{N}}\Sigma_{n}$ is measurable. Let now $Q_{0}$ be a countable dense subset of $\H$. Since $d^{\circ}_{C(t)}(u(t);\cdot)$ is $1$-Lipschitz, \eqref{eq:clarke-support} gives
\begin{align*}
\left\{(t,v)\colon v\in \partial d_{C(t)}(u(t))\right\}=\bigcap_{w\in Q_{0}}\left\{(t,v)\in I\times \H\colon \langle v,w\rangle-\Theta_{w}(t)\leq 0\right\}.
\end{align*}
Each map $(t,v)\mapsto \langle v,w\rangle-\Theta_{w}(t)$ is measurable in $t$ and continuous in $v$, hence $\mathcal{I}\otimes\mathcal{B}(\H)$-measurable.  Thus the above graph is $\mathcal{I}\otimes \mathcal{B}(\H)$-measurable. Since $\mathcal{I}$ is complete, $\H$ is Polish and the values are closed, the projection theorem \cite[Theorem~III.30]{Castaing_Valadier-1977} yields the measurability of $t\rightrightarrows \partial d_{C(t)}(u(t))$.\\
\noindent (ii) If $u(t)\in C(t)$, then \eqref{lem:trunc} gives $\partial d_{C(t)}(u(t))=N_{C(t)}(u(t))\cap \mathbb{B}$ and\begin{align*}
N_{C(t)}(u(t))=\bigcup_{n\in\mathbb{N}}n\,\partial d_{C(t)}(u(t)).
\end{align*}
Each map $t\rightrightarrows n\,\partial d_{C(t)}(u(t))$ is measurable by (i), and for every open $U\subset \H$ one has $\{t\colon N_{C(t)}(u(t))\cap U\neq \emptyset\}=\bigcup_{n}\{t\colon n\,\partial d_{C(t)}(u(t))\cap U\neq\emptyset\}\in \mathcal{I}$. The values are closed and contain $0$.
\end{proof}
\noindent The next result is a new enhanced version of the classical Gr\"onwall inequality, in the spirit of \cite[Theorems~3.1 and~3.2]{Vilches2024}, tailored to the quadratic estimates arising in Section~\ref{Main_result}.
\begin{proposition}[Enhanced Gr\"onwall inequality]\label{Gronwall}
Let $T>0$ and set $I:=[0,T]$. Let $\beta\colon I\to\mathbb{R}$ be a nonnegative absolutely continuous function, and let
$\varepsilon,K_i\colon I\to\mathbb{R}_{+}$, $i=1,\ldots,6$, be integrable functions. Suppose that, for a.e. $t\in I$,
\begin{equation}\label{eq:EG3-ineq}
\begin{aligned}
\dot{\beta}(t)
&\leq \varepsilon(t)
+K_1(t)\sqrt{\beta(t)}
+K_2(t)\beta(t) \\
&\quad
+K_3(t)\sqrt{\beta(t)}
\int_0^t K_4(s)\sqrt{\beta(s)}\,ds
+K_5(t)\int_0^t K_6(s)\sqrt{\beta(s)}\,ds.
\end{aligned}
\end{equation}
Then, for every $t\in I$,
\begin{align*}
\beta(t)
&\leq
\beta(0)
\exp\left(
\int_0^t\bigl(\eta_1(s)+\eta_2(s)\bigr)\,ds
\right)\\
&+\int_0^t
\bigl(\varepsilon(s)+\eta_1(s)\bigr)
\exp\left(
\int_s^t\bigl(\eta_1(\tau)+\eta_2(\tau)\bigr)\,d\tau
\right)\,ds,
\end{align*}
with {\small $\eta_1(t)
:=K_1(t)+K_5(t)\int_0^t K_6(s)\,ds$} and  {\small $\eta_2(t)
:=K_2(t)+K_3(t)\int_0^t K_4(s)\,ds$}.
\end{proposition}

\begin{proof}
Let $\Phi\colon I\to\mathbb{R}_{+}$ denote the right-hand side of
\eqref{eq:EG3-ineq}, that is,
\begin{align*}
\Phi(t)
&:=
\varepsilon(t)
+K_1(t)\sqrt{\beta(t)}
+K_2(t)\beta(t) \\
&\quad
+K_3(t)\sqrt{\beta(t)}
\int_0^t K_4(s)\sqrt{\beta(s)}\,ds
+K_5(t)\int_0^t K_6(s)\sqrt{\beta(s)}\,ds,
\end{align*}
and define
\[
v(t):=\beta(0)+\int_0^t\Phi(s)\,ds,
\qquad t\in I.
\]
Since $\beta$ is continuous on the compact interval $I$, it is bounded.
Therefore, the integrability of $\varepsilon$ and $K_1,\ldots,K_6$
implies that $\Phi\in L^1(I;\mathbb{R})$. Consequently, $v$ is well
defined and absolutely continuous, with $\dot v(t)=\Phi(t)$ for a.e. $t\in I$. The function $v$ satisfies:
\begin{enumerate}[label=(\roman*)]
    \item $\beta(0)=v(0)$;
    \item $\beta(t)\leq v(t)$ for every $t\in I$;
    \item $v$ is nondecreasing on $I$.
\end{enumerate}
Indeed, (i) follows from the definition of $v$. Moreover, by
the absolute continuity of $\beta$ and \eqref{eq:EG3-ineq},
\[
\beta(t)
=
\beta(0)+\int_0^t\dot{\beta}(s)\,ds
\leq
\beta(0)+\int_0^t\Phi(s)\,ds
=
v(t),
\]
which proves (ii). Finally, (iii) follows from $\Phi\geq0$. Since $\varepsilon$ and $K_1,\ldots,K_6$ are nonnegative, the right-hand side of \eqref{eq:EG3-ineq} is nondecreasing with respect to $\beta$. Hence, using (ii), and then (iii), which gives $\sqrt{v(s)}\leq\sqrt{v(t)}$ whenever $0\leq s\leq t$, we obtain, for a.e. $t\in I$,
\begin{align*}
\dot v(t)
&\leq
\varepsilon(t)
+K_1(t)\sqrt{v(t)}
+K_2(t)v(t) \\
&\qquad \qquad 
+K_3(t)\sqrt{v(t)}
\int_0^t K_4(s)\sqrt{v(s)}\,ds
+K_5(t)\int_0^t K_6(s)\sqrt{v(s)}\,ds \\
&\leq
\varepsilon(t)
+K_1(t)\sqrt{v(t)}
+K_2(t)v(t) \\
&\qquad \qquad \qquad
+K_3(t)v(t)\int_0^t K_4(s)\,ds
+K_5(t)\sqrt{v(t)}\int_0^t K_6(s)\,ds \\
&=
\varepsilon(t)
+\eta_1(t)\sqrt{v(t)}
+\eta_2(t)v(t)\\
&\leq \varepsilon(t)+\eta_1(t)
+\bigl(\eta_1(t)+\eta_2(t)\bigr)v(t),
\end{align*}
where we used that $\sqrt{v(t)}\leq 1+v(t)$, for $t\in I$. The functions $\eta_{1}$ and $\eta_{2}$ are nonnegative and integrable on $I$, so the classical Gr\"onwall inequality yields
\begin{align*}
v(t)
&\leq
v(0)
\exp\left(
\int_0^t\bigl(\eta_1(s)+\eta_2(s)\bigr)\,ds
\right)\\
&+\int_0^t
\bigl(\varepsilon(s)+\eta_1(s)\bigr)
\exp\left(
\int_s^t\bigl(\eta_1(\tau)+\eta_2(\tau)\bigr)\,d\tau
\right)\,ds.
\end{align*}
The conclusion follows from $v(0)=\beta(0)$ and $\beta(t)\leq v(t)$.
\end{proof}

\subsection{One-Sided Lipschitz Property and Examples}
In this section, we present the one-sided Lipschitz (OSL) property together with several examples illustrating this condition. Introduced by Donchev in \cite{Donchev_T-1991-Functional} (see also \cite{Donchev_T-2002,Donchev_Farkhi_Mordukhovich-2007}), this notion was subsequently used by Donchev and Farkhi in \cite{Donchev_T-Farkhi_E-1999-Approximations} to extend Filippov's theorem for general differential inclusions.

\begin{definition}\label{def:OSL}
We say that a set-valued map $F\colon I\times \H\rightrightarrows \H$ is \textnormal{One-Sided Lipschitz} (OSL) if for all $r>0$, there exists an integrable function $L_{F}^{r}\colon I \to \mathbb{R}$ such that for every $x,y\in r\mathbb{B}$, a.e. $t\in I$, and $v\in F(t,x)$, there exists $w\in F(t,y)$ such that
\begin{align*}
    \langle x-y,v-w\rangle \leq L_{F}^{r}(t)\Vert x-y\Vert^2.
\end{align*}
\end{definition}
\noindent We next present classes of OSL set-valued maps that arise naturally in controlled differential systems, possibly with nonconvex values and without  Lipschitz continuity with respect to the Hausdorff distance.
\begin{example}
Let $U$ be a nonempty set and let $p\colon I\times \H\times U\to \H$ be a given mapping. Assume that there exists an integrable function $L\colon I\to\mathbb{R}$ such that
\begin{align*}
    \langle x - y, p(t, x, u) - p(t, y, u) \rangle \leq L(t)\Vert x - y\Vert^2,
\end{align*}
for a.e. $t\in I$, for all $x,y\in\mathcal{H}$, and for every $u\in U$. Then the set-valued map
\begin{align*}
    F(t, x) = p(t, x, U) := \left\lbrace v \in \mathcal{H} \mid \exists u \in U \text{ such that } v = p(t, x, u) \right\rbrace
\end{align*}
is one-sided Lipschitz, with $L_{F}^{r}=L$ for every $r>0$. Indeed,  given $v=p(t,x,u)\in F(t,x)$, it suffices to take $w:=p(t,y,u)\in F(t,y)$.
\end{example}
\noindent We next show that subdifferentials of weakly convex functions provide another important class of OSL set-valued maps.
\begin{example}
Let $\phi\colon I \times \H\to \mathbb{R}$ and set $\phi_t(\cdot):=\phi(t,\cdot)$. Assume that there exists an integrable function $\sigma\colon  I\to\mathbb{R}_{+}$  such that $\phi_{t}(\cdot)+(\sigma(t)/2)\Vert \cdot\Vert^2$ is a continuous convex function for every $t\in I$. Consider
\begin{align*}
    F(t,x):=\partial \phi_t(x) := \partial\left(\phi_t(\cdot) +\frac{\sigma(t)}{2}\Vert \cdot\Vert^2\right)(x)-\sigma(t)x,
\end{align*}
where $\partial$ is the convex subdifferential. Hence, these sets are nonempty since every continuous convex function on $\H$ is everywhere subdifferentiable. Then $F(t,x)$ satisfies the OSL condition with $L_{F}^{r}(t)=\sigma(t)$ for every $r>0$.
\end{example}
\noindent We next observe that the OSL property is preserved under perturbations by monotone operators.
\begin{example}
Let $A\colon \H \rightrightarrows \H$ be a monotone operator with
$\operatorname{dom} A = \H$, i.e.,
\begin{align*}
\langle x-y,u-v\rangle \geq 0 \quad \textrm{ for all } (x,u)\in \gph A \textrm{ and } (y,v)\in \gph A,
\end{align*}
and let $f\colon \H\to\H$ be one-sided Lipschitz with constant $L\in \mathbb{R}$, that is,
$$\langle x-y, f(x)-f(y)\rangle\leq L\Vert x-y\Vert^{2} \textrm{ for all  } x,y\in \H.
$$ 
Then $F:=f-A$ satisfies the OSL condition with $L_{F}^{r}\equiv L$ for every $r>0$.
\end{example}
\noindent We next show that the OSL property is compatible with the normal-cone perturbations arising in sweeping processes over uniformly prox-regular moving sets.
\begin{example}
Define the multivalued map $G\colon I\times\mathcal H\rightrightarrows\mathcal{H}$ by
\begin{align*}
G(t,x):=-\kappa(t)\partial d_{C(t)}(x)+F(t,x),
\end{align*}
where $F\colon I\times\mathcal{H}\rightrightarrows\mathcal{H}$ has nonempty values and satisfies the OSL condition, $C\colon I\rightrightarrows\mathcal H$ is a set-valued map such that each set $C(t)$ is $\rho$-uniformly prox-regular (see Section~\ref{Math_Prel}), and $\kappa\colon I\to \mathbb{R}_{+}$ is integrable. Then $G$ satisfies the following (OSL)-type estimate along the moving set: for all $r>0$, a.e. $t\in I$, all $x,y\in C(t)\cap r\mathbb{B}$ and  $v\in G(t,x)$, there exists $w\in G(t,y)$ such that
\begin{align*}
\langle x-y, v-w\rangle\leq \left(L_{F}^{r}(t)+\frac{\kappa(t)}{\rho}\right)\Vert x-y\Vert^{2}.
\end{align*}
Indeed, let $x,y\in C(t)\cap r\mathbb{B}$ and $v\in G(t,x)$, and write $v=-\kappa(t)\xi_{x}+u_{x}$ with $\xi_{x}\in \partial d_{C(t)}(x)$ and $u_{x}\in F(t,x)$. By the OSL property of $F$, there exists $u_{y}\in F(t,y)$ such that 
$$
\langle x-y,u_{x}-u_{y}\rangle\leq L_{F}^{r}(t)\Vert x-y\Vert^{2}.
$$ Take any $\xi_{y}\in \partial d_{C(t)}(y)$ and set $w:=-\kappa(t)\xi_{y}+u_{y}\in G(t,y)$. Since $x,y\in C(t)$, \eqref{lem:trunc} yields $\xi_{x}\in N^{P}(C(t);x)\cap \mathbb{B}$ and $\xi_{y}\in N^{P}(C(t);y)\cap \mathbb{B}$, so that the $\frac{1}{\rho}$-hypomonotonicity property of Proposition~\ref{Prop2.1}(e) gives
\begin{align*}
\langle \xi_{x}-\xi_{y},x-y\rangle\geq -\frac{1}{\rho}\Vert x-y\Vert^{2}.
\end{align*}
Therefore,
\begin{align*}
\langle x-y, v-w\rangle&=-\kappa(t)\langle x-y,\xi_{x}-\xi_{y}\rangle+\langle x-y,u_{x}-u_{y}\rangle\\
&\leq \left(L_{F}^{r}(t)+\frac{\kappa(t)}{\rho}\right)\Vert x-y\Vert^{2}.
\end{align*}
\end{example}

\section{Technical Assumptions}\label{hipo-sol}

In this section, we collect the assumptions on the data of the Volterra sweeping process \eqref{Sweeping-Dif1}. Assumptions \ref{Hf} and \ref{Hg} impose Carath\'eodory-type regularity and linear growth on the single-valued perturbations. The multivalued perturbation $F$ is subject to \ref{H1F} and \ref{H2F}, which provide the measurability and closedness properties required for the existence of the measurable selections used throughout the paper (see Lemma~\ref{Lemma_exis} and Proposition~\ref{meassum}), and to \ref{H3F}, a linear growth condition. The one-sided Lipschitz condition \ref{HOSL} of Definition~\ref{def:OSL} plays a central role, since it underlies the Filippov-type estimates of Section~\ref{Main_result}. Finally, \ref{HC} gathers the geometric and regularity assumptions on the moving set, while \ref{H4F} and \ref{HC3} are compactness conditions, only one of which is required in our main results (see Theorem~\ref{Filippov_T}).

\begin{enumerate}[leftmargin=2.8em]
    \item[\namedlabel{Hf}{$(\mathsf{H}^{f_1})$}]  The function $f_1\colon I\times \H\to \H$ satisfies:
\end{enumerate}
\begin{enumerate}
        \item[(a)] For each $x\in \H$, the map $t \mapsto f_1(t,x)$ is measurable.

        \item[(b)]  For all $r>0$, there exists  $l_{r}^{1}\colon I\to \mathbb{R}_{+}$ integrable such that for a.e. $t\in I$
\begin{align*}
        \Vert f_1(t,x)-f_1(t,y)\Vert \leq l_{r}^{1}(t)\Vert x-y\Vert \textrm{ for all } x,y\in r\mathbb{B}.
\end{align*}

        \item[(c)]  There exists a nonnegative integrable function $a_{1}\colon I\to \mathbb{R}_{+}$ such that
\begin{align*}
        \Vert f_1(t,x)\Vert \leq a_{1}(t) (1 + \Vert x\Vert) \textrm{ for a.e. } t\in I \textrm{ and all } x\in \H.
\end{align*}
\end{enumerate}

\begin{enumerate}[leftmargin=2.8em]
    \item[\namedlabel{Hg}{$(\mathsf{H}^{f_2})$}]  Set $D:=\{(t,s)\in I\times I\colon s\leq t\}$. The function $f_2\colon D\times \H\to \H$ satisfies:
    \end{enumerate}
\begin{enumerate}
\item[(a)] For each $x\in \H$, the map $(t,s) \mapsto f_2(t,s,x)$ is measurable on $D$.

           \item[(b)] For all $r>0$, there exists an integrable function $l_{r}^{2}\colon I\to \mathbb{R}_+$ such that for a.e. $t\in I$ and all $s\in [0,t]$
\begin{align*}
        \Vert f_2(t,s,x)-f_2(t,s,y)\Vert \leq l_{r}^{2}(t)\Vert x-y\Vert \textrm{ for all } x,y\in r\mathbb{B}.
\end{align*}

\item[(c)] There exists a nonnegative function $a_{2}\in L^1(D)$ such that
      \begin{align*}
        \Vert f_2(t,s,x)\Vert \leq a_{2}(t,s)(1+\Vert x\Vert) \textrm{ for a.e. } (t,s)\in D \textrm{ and all } x\in \H.
      \end{align*}
    \end{enumerate}
\noindent Notice that the Lipschitz constant $l_{r}^{2}$ in \ref{Hg}(b) depends on $t$ only. This separated structure is what makes the enhanced Gr\"onwall inequality of Proposition~\ref{Gronwall} applicable, with $K_4\equiv K_6\equiv 1$, in the proof of Theorem~\ref{Filippov_T}.

\begin{enumerate}[leftmargin=2.65em]
    \item[\namedlabel{HF}{$(\mathsf{H}^{F})$}]   The set-valued map $F\colon I\times \H\rightrightarrows \H$ has nonempty, closed and convex values, and satisfies:
        \end{enumerate}
\begin{enumerate}[leftmargin=3.65em]
        \item[\namedlabel{H1F}{$(\mathsf{H}^{F}_{1})$}] The set-valued map $t\rightrightarrows \gph F(t,\cdot)$ is measurable, that is, for every (norm) open set $U\subset \H \times \H$, the set $\{ t \in I\colon \gph F(t,\cdot)\cap U\neq \emptyset \}$ belongs to $\mathcal{I}$.

        \item[\namedlabel{H2F}{$(\mathsf{H}^{F}_{2})$}] For a.e. $t\in I$, $\gph F(t,\cdot)$ is closed in $\mathcal{H}\times \mathcal{H}_w$.

        \item[\namedlabel{H3F}{$(\mathsf{H}^{F}_{3})$}] There exist nonnegative integrable functions $c$ and $m$ such that
        \begin{equation*}
            \begin{aligned}
                \Vert F(t,x)\Vert:=\sup\{\Vert w\Vert \colon w\in F(t,x)\}\leq c(t) \Vert x\Vert + m(t),
            \end{aligned}
        \end{equation*}
        for all $x\in \H$ and a.e. $t\in I$.

\end{enumerate}

\begin{enumerate}[leftmargin=3.85em]
    \item[\namedlabel{H4F}{$(\mathsf{H}^{F}_{\mathrm{comp}})$}]
For all $r>0$, there exists an integrable function $k_r\colon I\to \mathbb{R}_+$ such that for a.e. $t\in I$ and every nonempty $A\subset r\mathbb{B}$, one has
\begin{align*}
    \chi(F(t,A))\leq k_r(t)\chi(A),
\end{align*}
where $\chi$ is either the Kuratowski or the Hausdorff measure of noncompactness and $F(t,A):=\bigcup_{x\in A}F(t,x)$. Notice that, by \ref{H3F}, the set $F(t,A)$ is bounded for a.e. $t\in I$, so that $\chi(F(t,A))$ is finite.
\end{enumerate}

\begin{enumerate}[leftmargin=2.95em]
    \item[\namedlabel{HOSL}{$(\mathsf{H}^{F}_{\mathrm{osl}})$}] One-Sided Lipschitz property (see Definition~\ref{def:OSL}): For all $r > 0$, there exists an integrable function $L_{F}^{r}\colon I \to \mathbb{R}$ such that for every $x,y\in r\mathbb{B}$, a.e. $t\in I$, and $v\in F(t,x)$, there exists $w\in F(t,y)$ such that
\begin{align*}
    \langle x-y,v-w\rangle \leq L_{F}^{r}(t)\Vert x-y\Vert^2.
\end{align*}
\end{enumerate}

\begin{enumerate}[leftmargin=2.65em]
    \item[\namedlabel{HC}{$(\mathsf{H}^{C})$}]  The set-valued map $C\colon I \rightrightarrows \H$
 has nonempty, closed values which are uniformly $\rho$-prox-regular for some $\rho\in ]0,+\infty]$ independent of $t$, and there exists $L_{C}\geq 0$ such that
\begin{equation*}
\operatorname{Haus}(C(t),C(s))\leq L_{C} |t-s| \textrm{ for all } s,t\in I.
\end{equation*}
\end{enumerate}
\begin{enumerate}[leftmargin=3.85em]
  \item[\namedlabel{HC3}{$(\mathsf{H}^{C}_{\mathrm{comp}})$}] For all $t\in I$ and all $r>0$, the set $r\mathbb{B}\cap C(t)$ is compact.
\end{enumerate}

We close this section with some remarks. The first one shows that, under the above assumptions, all the superposition maps and integrals appearing in \eqref{Sweeping-Dif1} and in Sections~\ref{Sweeping-sec} and \ref{Main_result} are well defined.

\begin{remark}\label{rem:superposition}
Assume \ref{Hf} and \ref{Hg}, and let $u\colon I\to \H$ be measurable with $\Vert u(t)\Vert\leq R$ for a.e. $t\in I$ and some $R\geq 0$. Then the following assertions hold.
\begin{enumerate}[label=\textnormal{(\roman*)}]
\item The map $t\mapsto f_1(t,u(t))$ is measurable and belongs to $L^1(I;\H)$, with $\Vert f_1(t,u(t))\Vert\leq (1+R)\, a_1(t)$ for a.e. $t\in I$.
\item The map $(t,s)\mapsto f_2(t,s,u(s))$ is measurable on $D$ and satisfies
$$\Vert f_2(t,s,u(s))\Vert\leq (1+R)\, a_2(t,s) \textrm{ for a.e. } (t,s)\in D.$$ 
Hence, for a.e. $t\in I$, the map $s\mapsto f_2(t,s,u(s))$ belongs to $L^1([0,t];\H)$, the map $t\mapsto \int_0^t f_2(t,s,u(s))\,ds$ is measurable and
\begin{align*}
\int_0^T\left\Vert \int_0^t f_2(t,s,u(s))\,ds\right\Vert dt\leq (1+R)\Vert a_2\Vert_{L^1(D)}<+\infty.
\end{align*}
\item The function $\gamma$ defined in \eqref{eta-eq} is nonnegative, measurable and integrable, and $\int_0^T \gamma(t)\,dt=\Vert a_1\Vert_{L^1(I)}+\Vert a_2\Vert_{L^1(D)}$.
\end{enumerate}
Indeed, by \ref{Hf}(a) and \ref{Hf}(b), the map $f_1$ is measurable in $t$ for each fixed $x$ and continuous in $x$ for a.e. $t$. Hence it is a Carath\'eodory map and therefore $\mathcal{I}\otimes \mathcal{B}(\H)$-measurable (see, e.g., \cite{Castaing_Valadier-1977,Aubin_Frankowska_2009_book}), the separability of $\H$ being used here. Since $t\mapsto (t,u(t))$ is measurable, the composition $t\mapsto f_1(t,u(t))$ is measurable, and (i) follows from \ref{Hf}(c). The same argument applied to $f_2$ gives the measurability of $(t,s)\mapsto f_2(t,s,u(s))$ on $D$, while \ref{Hg}(c) yields the pointwise bound in (ii). Finally, the remaining assertions in (ii), as well as (iii), follow from the Fubini-Tonelli theorem, since $a_2\in L^1(D)$ and $a_2\geq 0$.
\end{remark}
The next remark discusses assumptions on the set-valued map $F$.
\begin{remark}\label{rem:HF} 
\begin{enumerate}[label=\textnormal{(\roman*)}]
\item Since the topology of $\H\times \H_w$ is coarser than the norm topology of $\H\times \H$, assumption \ref{H2F} implies that $\gph F(t,\cdot)$ is also closed in $\H\times\H$ for a.e. $t\in I$. Hence, $t\rightrightarrows \gph F(t,\cdot)$ is a closed-valued set-valued map between the separable metric spaces $I$ and $\H\times \H$, and \ref{H1F} is its usual measurability. In particular, Lemma~\ref{Lemma_exis} and Proposition~\ref{meassum} apply.
\item For a.e. $t\in I$ and every $x\in \H$, the set $F(t,x)$ is a section of $\gph F(t,\cdot)$ and is therefore weakly closed by \ref{H2F}. Being also convex, by \ref{HF}, and bounded, by \ref{H3F}, each set $F(t,x)$ is convex and weakly compact.
\item In the proofs, \ref{H2F} is used in the following sequential form: for a.e. $t\in I$, if $x_n\to x$ strongly in $\H$, $v_n\in F(t,x_n)$ and $v_n\rightharpoonup v$ weakly in $\H$, then $v\in F(t,x)$. Since bounded subsets of the separable Hilbert space $\H$ are weakly metrizable, this sequential form is equivalent to \ref{H2F} on bounded sets, which, by \ref{H3F}, is the only situation encountered in this paper.
\item By \ref{H1F} and Remark~\ref{rem:HF}(i), the set $\gph F:=\{(t,u,v)\in I\times \H\times \H\colon v\in F(t,u)\}$ is $\mathcal{I}\otimes \mathcal{B}(\H\times \H)$-measurable, as the graph of a measurable closed-valued set-valued map defined on $I$ with values in the Polish space $\H\times\H$.
\end{enumerate}
\end{remark}

\begin{remark}\label{rem:monotone-radii}
If $0<r\leq r'$, then any function admissible in \ref{Hf}(b),
\ref{Hg}(b), or \ref{HOSL} for the radius $r'$ is also admissible
for the radius $r$. Hence, without loss of generality, the corresponding
local moduli may be chosen nondecreasing with respect to the radius.
Indeed, replacing them by 
\begin{equation*}
\widehat l^{1}_{r}(t)
:=\sup_{\substack{j\in\mathbb{N}\\1\leq j\leq \lceil r\rceil}}
l^{1}_{j}(t),\quad  
\widehat l^{2}_{r}(t)
:=\sup_{\substack{j\in\mathbb{N}\\1\leq j\leq \lceil r\rceil}}
l^{2}_{j}(t),  \,\textrm{ and } \,
\widehat L_{F}^{r}(t)
:=\sup_{\substack{j\in\mathbb{N}\\1\leq j\leq \lceil r\rceil}}
L_{F}^{j}(t),
\end{equation*}
we obtain integrable admissible moduli such that, for a.e. $t\in I$,
the maps $r\mapsto \widehat l^{1}_{r}(t)$, $r\mapsto \widehat l^{2}_{r}(t)$, and $r\mapsto \widehat L_{F}^{r}(t)$ are nondecreasing.
\end{remark}

\begin{remark}\label{rem:OSL-normalization}
\begin{enumerate}[label=\textnormal{(\roman*)}]
\item The inequality in property \ref{HOSL} remains valid when $L_{F}^{r}$ is replaced by any larger integrable function. In particular, replacing $L_{F}^{r}$ by its positive part $(L_{F}^{r})^{+}:=\max\{L_{F}^{r},0\}$, we may and do assume, without loss of generality, that $L_{F}^{r}(t)\geq 0$ for a.e. $t\in I$ and all $r>0$.
This normalization is used in Section~\ref{Main_result}. Indeed,  the coefficients of the enhanced Gr\"onwall inequality of Proposition~\ref{Gronwall} are required to be nonnegative, while the one-sided Lipschitz condition itself allows $L_{F}^{r}$ to take negative values (as in the dissipative case).
\item Assumption \ref{HOSL} is strictly weaker than the Lipschitz continuity of $F(t,\cdot)$ with respect to the Hausdorff distance. Indeed, if for all $r>0$ there exists an integrable function $\kappa_r\colon I\to \mathbb{R}_+$ such that
\begin{align*}
\operatorname{Haus}(F(t,x),F(t,y))\leq \kappa_r(t)\Vert x-y\Vert \quad \textrm{ for a.e. } t\in I \textrm{ and } \forall x,y\in r\mathbb{B},
\end{align*}
then, given $v\in F(t,x)$, the projection $w:=\operatorname{proj}_{F(t,y)}(v)$ is well defined, because $F(t,y)$ is nonempty, closed and convex, and satisfies $\Vert v-w\Vert=d(v,F(t,y))\leq \kappa_r(t)\Vert x-y\Vert$. The Cauchy--Schwarz inequality then yields $\langle x-y,v-w\rangle\leq \kappa_r(t)\Vert x-y\Vert^{2}$, so that \ref{HOSL} holds with $L_{F}^{r}=\kappa_r$. The converse fails, as shown in Proposition~\ref{prop:fishery}.
\end{enumerate}
\end{remark}

\section{Reduction of Volterra Sweeping Processes}\label{Sweeping-sec}
In this section, we provide a reduction result for Volterra sweeping processes:
{\small
\begin{equation}\label{Sweeping-Dif}
\left\{
\begin{aligned}
\dot{x}(t)&\in -N_{C(t)}(x(t))+f_1(t,x(t))+\int_0^t f_2(t,s,x(s))\, ds\\
&+F(t,x(t) + h(t)) + g(t)\mathbb{B} & \textrm{ a.e. } t\in I,\\
x(0)&=x_0\in C(0),
\end{aligned}
\right.
\end{equation}}
where $C\colon I \tto \H$ is a set-valued map satisfying \ref{HC},  $g\in L^1(I;\mathbb{R}_{+})$ and $h\in L^{\infty}(I;\H)$ are given perturbations. The choice $h\equiv 0$ and $g\equiv 0$ gives back the Volterra sweeping process \eqref{Sweeping-Dif1}. We will prove that the above differential inclusion is related to the following unconstrained differential inclusion:
{\small\begin{equation}\label{Reduced}
\left\{
\begin{aligned}
\dot{x}(t)\in& -(L_{C} + \nu(t, h, g))\partial d_{C(t)}(x(t))+f_1(t,x(t))+\int_{0}^t f_2(t,s,x(s))\, ds \\
&+ F(t,x(t) + h(t)) + g(t)\mathbb{B} \qquad\qquad\qquad  \textrm{ a.e. }  t\in I,\\
x(0)=&x_0,
\end{aligned}
\right.
\end{equation}}
 where $L_C$ is given by assumption \ref{HC} and $\nu(\cdot, h, g)$ is defined as
{\small 
\begin{equation*}
\nu(t, h, g) := 2\left(c(t) + \gamma(t)\right)(\theta(t, h) + \omega(t, g)) + 2 \left(\gamma(t) + m(t) + c(t)\Vert h(t)\Vert\right) + 2g(t),
\end{equation*}}
and {\small 
\begin{equation}\label{eta-eq}
\begin{aligned}
\gamma(t)&:=a_{1}(t)+\int_0^t a_{2}(t,s)\, ds,\\
\theta(t, h)&:=r_{0} \exp(2\int_0^t (c(s) + \gamma(s)) ds)+\int_0^t \varepsilon(s, h)\exp\left(2\int_s^t (c(\tau) + \gamma(\tau)) d\tau\right) ds,\\
\varepsilon(t, h) &:= L_{C} + 2 \left(\gamma(t) + m(t) + c(t)\Vert h(t)\Vert\right),\\
\omega(t, g) &:= 2\int_0^t g(s)\exp\left(2\int_s^t (c(\tau) + \gamma(\tau)) d\tau\right) ds.
\end{aligned}
\end{equation}}
Here $r_{0}\geq 0$ is any constant with $\Vert x_{0}\Vert \leq r_{0}$. The freedom in the choice of $r_{0}$ is needed in Section~\ref{Main_result}, where the same functions $\theta(\cdot,h)$, $\omega(\cdot,g)$ and $\nu(\cdot,h,g)$ have to bound simultaneously two trajectories issued from two different initial conditions.

The following elementary consequence of \eqref{eta-eq} explains the definition of $\nu(\cdot,h,g)$ and will be used repeatedly.

\begin{remark}\label{rem:theta-omega}
Let $g\in L^1(I;\mathbb{R}_{+})$ and $h\in L^{\infty}(I;\H)$, and set
\begin{align*}
R(t,h,g):=\theta(t,h)+\omega(t,g), \qquad t\in I.
\end{align*}
By Remark~\ref{rem:superposition}, the function $\gamma$ is nonnegative and integrable on $I$. Hence, so are $\varepsilon(\cdot,h)$ and $\nu(\cdot,h,g)$, and all the functions in \eqref{eta-eq} are well defined and absolutely continuous. Moreover, $R(\cdot,h,g)$ is the unique absolutely continuous solution of the linear differential equation:  $R(0,h,g)=r_{0}$ and 
\begin{align*}
\frac{d}{dt}R(t,h,g)=2\left(c(t)+\gamma(t)\right)R(t,h,g)+\varepsilon(t,h)+2g(t) \quad \textrm{ a.e. } t\in I.
\end{align*}
By the very definitions of $\nu(\cdot,h,g)$ and $\varepsilon(\cdot,h)$,
\begin{align}\label{eq:nu-derivative}
\frac{d}{dt}R(t,h,g)=L_{C}+\nu(t,h,g)\quad \textrm{ a.e. } t\in I.
\end{align}
In particular, $R(\cdot,h,g)$ is nonnegative and nondecreasing. Moreover, since $\varepsilon(t,h)\geq \varepsilon(t,0)$ and $g\geq 0$, the comparison principle for linear differential equations gives
\begin{align}\label{eq:R-monotone}
R(t,0,0)\leq R(t,h,g) \quad \forall t\in I, \, \textrm{ and }\, \nu(t,0,0)\leq \nu(t,h,g)\ \textrm{ for a.e. } t\in I.
\end{align}
Identity \eqref{eq:nu-derivative} is the guiding principle behind \eqref{Reduced}. Moreover, as shown in Proposition~\ref{Main_Result_Red} below, $R(\cdot,h,g)$ bounds the norm of any solution of \eqref{Sweeping-Dif}, while $L_{C}+\nu(\cdot,h,g)$ bounds its velocity. The coefficient in front of $\partial d_{C(t)}(\cdot)$ in \eqref{Reduced} is thus large enough to force the trajectories of the reduced inclusion to remain in the moving set.
\end{remark}

We will also use the following Gr\"onwall-type inequality adapted to the Volterra structure of \eqref{Sweeping-Dif}, which follows from the classical Gr\"onwall inequality after passing to the nondecreasing envelope of $u$ (see also \cite[Theorem 3.1]{Vilches2024}).

\begin{lemma}\label{lem:volterra-gronwall}
Assume \ref{Hf}, \ref{Hg} and \ref{H3F}. Let $\lambda>0$, $u_0\geq 0$, $\beta\in L^1(I;\mathbb{R}_{+})$ and let $u\colon I\to \mathbb{R}_{+}$ be a continuous function such that
{\small 
\begin{align}\label{eq:volterra-gronwall}
u(t)\leq u_0+\int_0^t \beta(s)\, ds+\lambda\int_0^t\left(\left(a_1(s)+c(s)\right)u(s)+\int_0^s a_{2}(s,\tau)u(\tau)\, d\tau\right)\, ds \,  \forall t\in I.
\end{align}}
Then, for all $t\in I$, one has $u(t)\leq w(t)$, where
{\small 
\begin{align*}
w(t):=u_0\exp\left(\lambda\int_0^t \left(c(s)+\gamma(s)\right)\, ds\right)+\int_0^t \beta(s)\exp\left(\lambda\int_s^t \left(c(\tau)+\gamma(\tau)\right)\, d\tau\right)\, ds
\end{align*}}
is the solution of $\dot{w}(t)=\lambda\left(c(t)+\gamma(t)\right)w(t)+\beta(t)$ a.e. $t\in I$, with $w(0)=u_0$.
\end{lemma}
\begin{proof}
Let $\Lambda(t):=\max_{s\in [0,t]}u(s)$, which is continuous and nondecreasing on $I$. Since the right-hand side of \eqref{eq:volterra-gronwall} is nondecreasing with respect to $t$, the inequality \eqref{eq:volterra-gronwall} holds with $u(t)$ replaced by $\Lambda(t)$ on its left-hand side. Moreover, $u(s)\leq \Lambda(s)$ for all $s\in I$ and $u(\tau)\leq \Lambda(\tau)\leq \Lambda(s)$ whenever $0\leq \tau\leq s$, because $\Lambda$ is nondecreasing. Hence, by virtue of the definition of $\gamma$ in \eqref{eta-eq}, for all $t\in I$,
\begin{align*}
\Lambda(t)&\leq u_0+\int_0^t \beta(s)\, ds+\lambda\int_0^t\left(a_1(s)+c(s)+\int_0^s a_{2}(s,\tau)\, d\tau\right)\Lambda(s)\, ds\\
&=u_0+\int_0^t \beta(s)\, ds+\lambda\int_0^t\left(c(s)+\gamma(s)\right)\Lambda(s)\, ds.
\end{align*}
Since $c+\gamma$ is nonnegative and integrable (see Remark~\ref{rem:superposition}) and $t\mapsto u_0+\int_0^t \beta(s)\, ds$ is nondecreasing and absolutely continuous, the classical Gr\"onwall inequality yields  $u(t)\leq \Lambda(t)\leq w(t)$ for all $t\in I$.
\end{proof}

The idea of reducing a sweeping process to an unconstrained differential inclusion governed by the subdifferential of the distance function to the moving set goes back to Haddad, Jourani and Thibault \cite{Haddad2009}. The next result adapts it to the Volterra setting and to the presence of the perturbations $h$ and $g$. We establish the equivalence of the systems \eqref{Reduced} and \eqref{Sweeping-Dif}.
\begin{proposition}\label{Main_Result_Red}
Assume, in addition to \ref{HC}, \ref{HF}, \ref{Hg} and \ref{Hf}, that $x_{0}\in C(0)$, and let $r_{0}\geq \Vert x_{0}\Vert$ be defined in \eqref{eta-eq}. Let $x\colon I \to \H$ be an absolutely continuous function. Then, for every $g\in L^1(I;\mathbb{R}_{+})$ and every $h\in L^{\infty}(I;\H)$, $x(\cdot)$ is a solution of the differential inclusion given by \eqref{Reduced} if and only if it is a solution of \eqref{Sweeping-Dif}. Moreover, in any case,  $x(t)\in C(t)$ for all $t\in I$ and
\begin{equation*}
\begin{aligned}
  \Vert x(t)\Vert &\leq \theta(t, h) + \omega(t, g) & \textrm{ for all } t\in I,\\
  \Vert \dot{x}(t)\Vert &\leq 2 \left(\gamma(t) + c(t)\right) (\theta(t, h) + \omega(t, g)) + \varepsilon(t, h) + 2g(t)& \textrm{ a.e. } t\in I.
\end{aligned}
\end{equation*}
\end{proposition}
\begin{proof}
Throughout the proof we abbreviate $R(t):=R(t,h,g)=\theta(t,h)+\omega(t,g)$ and $\nu(t):=\nu(t,h,g)$, and we use Remark~\ref{rem:theta-omega}. In particular, $R$ is nondecreasing, $R(0)=r_{0}\geq \Vert x_0\Vert$ and $\dot{R}=L_{C}+\nu$ a.e. on $I$. We also set, for a.e. $t\in I$,
\begin{equation*}
\pi(t):=\gamma(t)+m(t)+c(t)\Vert h(t)\Vert+g(t)\geq 0,
\end{equation*}
so that $\varepsilon(t,h)=L_{C}+2\left(\pi(t)-g(t)\right)$ and $\nu(t)=2\left(c(t)+\gamma(t)\right)R(t)+2\pi(t)$.

\noindent \textbf{Step 1 (measurable selections).} Assume that $x(\cdot)$ solves \eqref{Reduced}. Since $x$ and $h$ are measurable, Lemma~\ref{lem:meas-subdiff}(i) shows that $\Gamma_{1}(t):=-(L_{C}+\nu(t))\partial d_{C(t)}(x(t))$ is a measurable set-valued map with nonempty closed values, while Lemma~\ref{Lemma_exis} shows that $t\rightrightarrows F(t,x(t)+h(t))$ is measurable with nonempty closed values. Hence so is $\Gamma_{2}(t):=F(t,x(t)+h(t))+g(t)\mathbb{B}$, whose graph is therefore $\mathcal{I}\otimes\mathcal{B}(\H)$-measurable. By Remark~\ref{rem:superposition}, the function
\begin{align*}
u(t):=\dot{x}(t)-f_{1}(t,x(t))-\int_{0}^{t}f_{2}(t,s,x(s))\, ds
\end{align*}
is a measurable selection of $\Gamma_{1}+\Gamma_{2}$. Proposition~\ref{meassum}, applied first to the pair $(\Gamma_{1},\Gamma_{2})$ and then to the pair $\bigl(t\rightrightarrows F(t,x(t)+h(t)),\ t\rightrightarrows g(t)\mathbb{B}\bigr)$, provides measurable selections $d(\cdot)$, $f(\cdot)$ and $b(\cdot)$ of $t\rightrightarrows \partial d_{C(t)}(x(t))$, $t\rightrightarrows F(t,x(t)+h(t))$ and $t\rightrightarrows \mathbb{B}$, respectively, such that
\begin{align}\label{eq:decomp}
\dot{x}(t)=-(L_C+\nu(t))d(t)+\ell(t) \quad \textrm{ a.e. } t\in I,\\
\textrm{ where } \quad \ell(t):=f_1(t,x(t))+\int_0^t f_2(t,s,x(s))\, ds + f(t) + g(t)b(t).\nonumber
\end{align}
(When $L_{C}+\nu(t)=0$ we simply set $d(t):=0\in \partial d_{C(t)}(x(t))$ if $x(t)\in C(t)$, and we pick any measurable selection otherwise. Both choices are compatible with \eqref{eq:decomp}.)

\noindent \textbf{Step 2 (a priori bound for \eqref{Reduced}).} One has $\partial d_{C(t)}(y)\subset \mathbb{B}$ for every $y\in \H$, so that $\Vert d(t)\Vert \leq 1$. Hence, by \ref{Hf}(c), \ref{Hg}(c) and \ref{H3F}, for a.e. $t\in I$,
\begin{align}\label{eq:ell-bound}
\Vert \ell(t)\Vert \leq \pi(t)+\left(a_{1}(t)+c(t)\right)\Vert x(t)\Vert +\int_0^t a_{2}(t,s)\Vert x(s)\Vert\, ds,
\end{align}
where we have used that $a_1(t)+\int_0^t a_2(t,s)\, ds=\gamma(t)$ and that $\Vert x(t)+h(t)\Vert \leq \Vert x(t)\Vert+\Vert h(t)\Vert$, and therefore $\Vert \dot{x}(t)\Vert \leq L_{C}+\nu(t)+\Vert \ell(t)\Vert$. Integrating on $[0,t]$ and using that $x(0)=x_0$ and $\Vert x_{0}\Vert \leq r_{0}$, we obtain, for all $t\in I$,
\begin{align*}
\Vert x(t)\Vert &\leq r_{0}+\int_0^t \left(L_{C}+\nu(s)+\pi(s)\right)\, ds\\
&\quad +\int_0^t\left(\left(a_{1}(s)+c(s)\right)\Vert x(s)\Vert+\int_0^s a_{2}(s,\tau)\Vert x(\tau)\Vert\, d\tau\right)\, ds.
\end{align*}
Lemma~\ref{lem:volterra-gronwall}, used with $\lambda=1$, $u=\Vert x(\cdot)\Vert$, $u_0=r_{0}$ and $\beta=L_{C}+\nu+\pi$, yields
\begin{align}\label{eq:apriori-q}
\Vert x(t)\Vert \leq q(t)\quad \textrm{ for all } t\in I,
\end{align}
where $q$ is the unique absolutely continuous solution of $\dot{q}=\left(c+\gamma\right)q+L_{C}+\nu+\pi$ on $I$ with $q(0)=r_{0}$. We claim that
\begin{align}\label{eq:apriori-2R}
q(t)\leq 2R(t) \quad \textrm{ for all } t\in I.
\end{align}
Indeed, set $z:=2R-q$, which is absolutely continuous with $z(0)=2r_{0}-r_{0}=r_{0}\geq 0$. By \eqref{eq:nu-derivative} and the identity $\nu=2(c+\gamma)R+2\pi$, we get, for a.e. $t\in I$,
\begin{align*}
\dot{z}(t)&=2\left(L_{C}+\nu(t)\right)-\left(c(t)+\gamma(t)\right)q(t)-\left(L_{C}+\nu(t)+\pi(t)\right)\\
&=L_{C}+\nu(t)-\left(c(t)+\gamma(t)\right)q(t)-\pi(t)\\
&=\left(c(t)+\gamma(t)\right)\left(2R(t)-q(t)\right)+L_{C}+\pi(t)\\
&=\left(c(t)+\gamma(t)\right)z(t)+L_{C}+\pi(t).
\end{align*}
Since $z(0)\geq 0$ and $L_{C}+\pi\geq 0$, the variation of constants formula gives
{ \small 
\begin{align*}
z(t)=z(0)\exp\left(\int_0^t \left(c+\gamma\right)\right)+\int_0^t \left(L_{C}+\pi(s)\right)\exp\left(\int_s^t \left(c+\gamma\right)\right)\, ds\geq 0 \quad  \forall t\in I,
\end{align*}}
which proves \eqref{eq:apriori-2R}. Combining \eqref{eq:apriori-q} and \eqref{eq:apriori-2R}, we conclude that
\begin{align}\label{eq:apriori-final}
\Vert x(t)\Vert \leq 2R(t)=2\left(\theta(t,h)+\omega(t,g)\right)\quad \textrm{ for all } t\in I.
\end{align}
Let us stress that the sharper bound $\Vert x(t)\Vert \leq R(t)$ is not available at this stage. It will be obtained a posteriori, in Step~5.

\noindent \textbf{Step 3 (the trajectory remains in the moving set).} Now, let us consider the function $\varphi(t):=\frac{1}{2}d_{C(t)}^2 (x(t))$ defined on the interval $[0,t^{*}]$, where $t^{*}:=\inf \{t\in I \colon \varphi(t)\geq \frac{\rho^2}{2}\}$ with the convention $\inf \emptyset:=T$ (so that $t^{*}=T$ when $\rho=+\infty$). The function $t\mapsto d_{C(t)}(x(t))$ is absolutely continuous, since \ref{HC} and the $1$-Lipschitz continuity of the distance function give
\begin{align*}
\left|d_{C(t)}(x(t))-d_{C(s)}(x(s))\right|\leq \Vert x(t)-x(s)\Vert+L_{C}|t-s| \quad \textrm{ for all } s,t\in I,
\end{align*}
and hence so is $\varphi$. Since $x(0)=x_0\in C(0)$, it is clear that $t^{*}>0$, and $d_{C(t)}(x(t))<\rho$ for all $t\in [0,t^{*}[$. Hence, for a.e. $t\in [0,t^{*}[$, one has
\begin{equation*}
   \begin{aligned}
    \dot{\varphi}(t)&\leq L_C d_{C(t)}(x(t))+\langle x(t)-\operatorname{proj}_{C(t)}(x(t)),\dot{x}(t)\rangle\\
    &= L_C d_{C(t)}(x(t)) -(L_C+\nu(t))\langle x(t)-\operatorname{proj}_{C(t)}(x(t)),d(t)\rangle\\
    &+\langle x(t)-\operatorname{proj}_{C(t)}(x(t)),\ell(t)\rangle.
    \end{aligned}
\end{equation*}
The first inequality follows from \eqref{eq:decomp} and from the following two facts. On the one hand, by Corollary~\ref{cor:squared-distance}, the function $\frac{1}{2}d^2_{C(t)}(\cdot)$ is differentiable on the open set $\{y\in \H\colon d_{C(t)}(y)<\rho\}$. On the other hand, for $\delta>0$, assumption \ref{HC} gives $d_{C(t+\delta)}(y)\leq d_{C(t)}(y)+L_{C}\delta$ for every $y\in \H$, whence
{\small 
\begin{align*}
\varphi(t+\delta)-\varphi(t)\leq \frac{1}{2}\left(d^2_{C(t)}(x(t+\delta))-d^2_{C(t)}(x(t))\right)+L_{C}\delta\, d_{C(t)}(x(t+\delta))+\frac{L_{C}^{2}\delta^{2}}{2}.
\end{align*}}
Dividing by $\delta>0$ and letting $\delta\downarrow 0$ yields the announced estimate at every point $t\in [0,t^{*}[$ where $x$ and $\varphi$ are differentiable. Next, we claim that
$$
\langle x(t)-\operatorname{proj}_{C(t)}(x(t)),d(t)\rangle=d_{C(t)}(x(t)) \quad \textrm{ for all } t\in [0,t^{*}[.
$$
Indeed, if $d_{C(t)}(x(t))=0$, then $x(t)=\operatorname{proj}_{C(t)}(x(t))$ and both sides vanish, while if $0<d_{C(t)}(x(t))<\rho$, then Proposition~\ref{Prop2.1}(c) and (d) imply that 
$$
\partial d_{C(t)}(x(t))=\frac{x(t)-\operatorname{proj}_{C(t)}(x(t))}{d_{C(t)}(x(t))},
$$
 so that the claim follows from $\Vert x(t)-\operatorname{proj}_{C(t)}(x(t))\Vert=d_{C(t)}(x(t))$.

Therefore, $\Vert x(t)-\operatorname{proj}_{C(t)}(x(t))\Vert=d_{C(t)}(x(t))$ and the estimates \eqref{eq:ell-bound} and \eqref{eq:apriori-final} together with the monotonicity of $R$, we obtain, for a.e. $t\in [0,t^{*}[$,
\begin{align*}
\Vert \ell(t)\Vert &\leq \pi(t)+\left(a_{1}(t)+c(t)\right)2R(t)+\int_0^t a_{2}(t,s)2R(s)\, ds\\
&\leq \pi(t)+2\left(\gamma(t)+c(t)\right)R(t)=\nu(t)-\pi(t),
\end{align*}
and consequently
\begin{align*}
\dot{\varphi}(t)&\leq \left[L_{C}-(L_{C}+\nu(t))+\Vert \ell(t)\Vert\right]d_{C(t)}(x(t))\leq -\pi(t)\, d_{C(t)}(x(t))\leq 0.
\end{align*}
Hence $\varphi$ is nonincreasing on $[0,t^{*}[$ and, since $\varphi(0)=0$ and $\varphi\geq 0$, we get $\varphi\equiv 0$ on $[0,t^{*}[$. If we had $t^{*}<T$, then $\rho<+\infty$ and, by the definition of $t^{*}$, there would exist a sequence $t_{n}\downarrow t^{*}$ with $\varphi(t_{n})\geq \rho^{2}/2$. The continuity of $\varphi$ would then give both $\varphi(t^{*})\geq \rho^{2}/2>0$ and $\varphi(t^{*})=\lim_{t\uparrow t^{*}}\varphi(t)=0$, a contradiction. Thus $t^{*}=T$ and $\varphi\equiv 0$ on $I$, i.e., $ x(t)\in C(t)$ for all $t\in I$.

\noindent \textbf{Step 4 (from \eqref{Reduced} to \eqref{Sweeping-Dif}).} Since $x(t)\in C(t)$ for all $t\in I$, \eqref{lem:trunc} gives
\begin{align*}
    \partial d_{C(t)}(x(t))=N_{C(t)}(x(t))\cap \mathbb{B} \textrm{ for all } t\in I.
\end{align*}
Hence, we deduce that $-(L_{C}+\nu(t))\partial d_{C(t)}(x(t))\subset -N_{C(t)}(x(t))$ for a.e. $t\in I$. Hence, $x(\cdot)$ is a solution of \eqref{Sweeping-Dif}, and the estimates follow from Step~5 below.

\noindent \textbf{Step 5 (from \eqref{Sweeping-Dif} to \eqref{Reduced}, and the estimates).} Reciprocally, assume that $x(\cdot)$ solves \eqref{Sweeping-Dif}. Since $N_{C(t)}(y)=\emptyset$ whenever $y\notin C(t)$, the inclusion in \eqref{Sweeping-Dif} forces $x(t)\in C(t)$ for a.e. $t\in I$, hence for all $t\in I$ by the continuity of $x$, the closedness of the values of $C$ and \ref{HC}. Thus, by Lemma~\ref{lem:meas-subdiff}, both $t\rightrightarrows N_{C(t)}(x(t))$ and $t\rightrightarrows \partial d_{C(t)}(x(t))$ are measurable with nonempty closed values, and, arguing exactly as in Step~1 (with $\Gamma_{1}(t):=-N_{C(t)}(x(t))$), Proposition~\ref{meassum} provides  selections $f(\cdot)$ of $t\rightrightarrows F(t,x(t) + h(t))$ and $b(\cdot)$ of $t\rightrightarrows \mathbb{B}$ such that
\begin{align*}
\dot{x}(t)\in -N_{C(t)}(x(t)) + \ell(t)\quad \textrm{ a.e. } t\in I,
\end{align*}
where $\ell(t):=f_1(t,x(t))+\int_0^t f_2(t,s,x(s))\, ds + f(t) + g(t) b(t)$. Then, according  to \cite[Proposition~1]{MR2179241}, one has
\begin{align*}
\Vert \dot{x}(t) - \ell(t)\Vert \leq L_C+\Vert \ell(t)\Vert \quad \textrm{ a.e. } t\in I.
\end{align*}
Then, by  \ref{Hf}, \ref{Hg} and \ref{HF}, the bound \eqref{eq:ell-bound} holds and yields, for a.e. $t\in I$,
\begin{equation*}
\begin{aligned}
\Vert \dot{x}(t)\Vert &\leq L_C + 2\Vert \ell(t)\Vert \\
&\leq L_{C}+2\pi(t)+2\left(\left(a_{1}(t)+c(t)\right)\Vert x(t)\Vert+\int_0^t a_{2}(t,s)\Vert x(s)\Vert\, ds\right)\\
&{}=\varepsilon(t,h)+2g(t)+2\left(\left(a_{1}(t)+c(t)\right)\Vert x(t)\Vert+\int_0^t a_{2}(t,s)\Vert x(s)\Vert\, ds\right).
\end{aligned}
\end{equation*}
Integrating on $[0,t]$, using that $\Vert x_{0}\Vert \leq r_{0}$ and applying Lemma~\ref{lem:volterra-gronwall} with $\lambda=2$, $u=\Vert x(\cdot)\Vert$, $u_0=r_{0}$ and $\beta=\varepsilon(\cdot,h)+2g$, we obtain from Remark~\ref{rem:theta-omega} that
\begin{align*}
    \Vert x(t)\Vert \leq R(t)=\theta(t, h) + \omega(t, g)\quad \textrm{ for all } t\in I,
\end{align*}
which is the first estimate of the statement. Finally, using \eqref{eq:ell-bound}, the bound $\Vert x(s)\Vert \leq R(s)\leq R(t)$ for $s\leq t$ and the definition of $\nu$, we get, for a.e. $t\in I$,
\begin{equation*}
\Vert \dot{x}(t)-\ell(t)\Vert \leq L_C+\Vert \ell(t)\Vert
\leq L_{C}+\pi(t)+\left(\gamma(t)+c(t)\right)R(t)
\leq L_C + \nu(t, h, g),
\end{equation*}
as well as $\Vert \dot{x}(t)\Vert \leq L_{C}+2\Vert \ell(t)\Vert \leq L_{C}+2\pi(t)+2(\gamma(t)+c(t))R(t)=L_{C}+\nu(t,h,g)=2\left(\gamma(t)+c(t)\right)R(t)+\varepsilon(t,h)+2g(t)$, which is the second estimate of the statement.

To conclude, set $\zeta(t):=\ell(t)-\dot{x}(t)\in N_{C(t)}(x(t))$, so that $\Vert \zeta(t)\Vert \leq L_{C}+\nu(t)$ for a.e. $t\in I$. If $L_{C}+\nu(t)>0$, then $\zeta(t)/(L_{C}+\nu(t))\in N_{C(t)}(x(t))\cap \mathbb{B}=\partial d_{C(t)}(x(t))$, by \eqref{lem:trunc}, and hence $\zeta(t)\in (L_{C}+\nu(t))\partial d_{C(t)}(x(t))$. If $L_{C}+\nu(t)=0$, then $\zeta(t)=0\in (L_{C}+\nu(t))\partial d_{C(t)}(x(t))$, because $x(t)\in C(t)$ implies $0\in \partial d_{C(t)}(x(t))$. In both cases,
\begin{align*}
\dot{x}(t)=-\zeta(t)+\ell(t)&\in -(L_{C}+\nu(t,h,g))\partial d_{C(t)}(x(t))+f_1(t,x(t))\\
&\quad +\int_0^t f_2(t,s,x(s))\, ds+F(t,x(t)+h(t))+g(t)\mathbb{B}
\end{align*}
for a.e. $t\in I$, i.e.,  $x(\cdot)$ is a solution of \eqref{Reduced}. The proof is finished.
\end{proof}

\section{Filippov's Theorem for Volterra Sweeping Processes}\label{Main_result}

In the following, we prove the main result of this section. Recall from Remark~\ref{rem:theta-omega} that $R(t,h,g)=\theta(t,h)+\omega(t,g)$ and that $L_{C}+\nu(t,h,g)=\frac{d}{dt}R(t,h,g)$. Throughout this section, $R(\cdot,h,g)$ and $L_{C}+\nu(\cdot,h,g)$ play the role of the state bound and of the velocity bound provided by Proposition~\ref{Main_Result_Red}.

\begin{theorem}\label{Filippov_T} Let $g\in L^1(I;\mathbb{R}_{+})$ and $h\in L^{\infty}(I;\H)$. Assume, in addition to \ref{HC}, \ref{HF},
\ref{Hf}, \ref{Hg} and \ref{HOSL}, that one of the following assumptions holds:
\begin{enumerate}[label=\textnormal{(\roman*)}]
    \item \ref{HC3}
    \item \ref{H4F}.
\end{enumerate}
Let $y(\cdot)$ be a solution  of the following differential inclusion:
\begin{equation}\label{Filippov1}
\left\{
\begin{aligned}
\dot{y}(t)&\in  -N_{C(t)}(y(t))+f_1(t,y(t))+\int_0^t f_2(t,s,y(s))\, ds\\
& + F(t,y(t) + h(t)) +  g(t)\mathbb{B} & \textrm{ a.e. } t\in I,\\
y(0)&=y_0\in C(0).
\end{aligned}
\right.
\end{equation}
Then, for every nonempty compact set $\mathcal{K}\subset C(0)$ and every $x_0\in \operatorname{Proj}_{\mathcal{K}}(y_0)$, and provided that the constant $r_{0}$ of \eqref{eta-eq} satisfies
\begin{align}\label{eq:r0}
r_{0}\geq \max\left\{\Vert x_{0}\Vert, \Vert y_{0}\Vert\right\},
\end{align}
there exists an absolutely continuous solution $x(\cdot)$ of \eqref{Sweeping-Dif1} with $x(0)=x_{0}$,
and the following estimates hold, where $\overline{R}:=R(T,h,g)+\Vert h\Vert_{\infty}$:
\begin{enumerate}
\item[\textnormal{(a)}] For all $t\in I$,
{\small\begin{align*}
\Vert x(t)-y(t)\Vert^{2}\leq &\operatorname{dist}^{2}(y_0, \mathcal{K})\times\exp\left(\int_0^t \left(\eta_{1}(s) + \eta_{2}(s)\right)\, ds\right)\\
   \qquad\quad &+2\int_0^t\left[\mu(s) + \eta_{1}(s)\right]\times\exp\left(\int_s^t \left(\eta_{1}(\tau) + \eta_{2}(\tau)\right)\, d\tau\right)\, ds,
    \end{align*}}
where
\begin{align*}
\mu(t) & := L_{F}^{\overline{R}}(t)\Vert h(t)\Vert^{2}+\left(g(t)+4\left(L_{C}+\nu(t,h,g)\right)\right)\Vert h(t)\Vert,\\
\eta_{1}(t) &:= \sqrt{2}\left(2L_{F}^{\overline{R}}(t)\Vert h(t)\Vert + l_{\overline{R}}^{1}(t)\Vert h(t)\Vert + g(t)\right) + \sqrt{2}\ l_{\overline{R}}^{2}(t)\Vert h(t)\Vert\, t,\\
\eta_{2}(t)&:=  2\left(L_{F}^{\overline{R}}(t) + l_{\overline{R}}^{1}(t) +\frac{L_{C} + \nu(t, h, g)}{\rho}\right) + 2\ l_{\overline{R}}^{2}(t)\, t.
\end{align*}

\item[\textnormal{(b)}] If $h\equiv 0$, one has for all $t\in I$
{\small\begin{align*}
       \Vert x(t)-y(t)\Vert\leq \operatorname{dist}(y_0, \mathcal{K})\times\exp\left(\int_0^t \eta(s)\, ds\right) + \int_0^t g(s)\exp\left(\int_s^t \eta(\tau)\, d\tau\right)\, ds,
    \end{align*}}
where $ \eta(t):= L_{F}^{\overline{R}}(t) + l_{\overline{R}}^{1}(t) +\frac{L_{C} + \nu(t, 0, g)}{\rho} + l_{\overline{R}}^{2}(t)\, t$.
\end{enumerate}
\end{theorem}

\begin{remark}
Since $x_{0}\in \operatorname{Proj}_{\mathcal{K}}(y_{0})$, one has $\Vert x(0)-y(0)\Vert = \operatorname{dist}(y_{0},\mathcal{K})$, which is why the initial distance appears on the right-hand side of both estimates. Condition \eqref{eq:r0} only requires the constant $r_{0}$ of \eqref{eta-eq} to dominate the norms of the two initial conditions, which is all that is used in the proof (Proposition~\ref{Main_Result_Red} is applied to $y(\cdot)$, issued from $y_{0}$, and to $x(\cdot)$, issued from $x_{0}$). It is guaranteed, independently of the choice of $x_{0}\in \operatorname{Proj}_{\mathcal{K}}(y_{0})$, by either of the conditions
\begin{align*}
r_{0}\geq \Vert y_{0}\Vert+\operatorname{dist}(y_{0},\mathcal{K}) \qquad \textrm{ or } \qquad r_{0}\geq \max\left\{\Vert y_{0}\Vert, \max_{z\in \mathcal{K}}\Vert z\Vert\right\},
\end{align*}
the second one being the choice used in Section~\ref{A_Opt_Cont}. Note also that, when $h\equiv 0$ and $g\equiv 0$, one has $\mu\equiv 0$ and $\eta_{1}\equiv 0$, so that $\eta_{2}=2\eta$ and the estimate in \textnormal{(a)} reduces exactly to the square of the estimate in \textnormal{(b)}.
\end{remark}
\begin{proof}
Fix a nonempty compact set $\mathcal{K}\subset C(0)$ and set $x_0\in\operatorname{Proj}_{\mathcal{K}}(y_0)$, so that $\Vert x_{0}-y_{0}\Vert=\operatorname{dist}(y_{0},\mathcal{K})$ and, by \eqref{eq:r0}, $\max\{\Vert x_{0}\Vert,\Vert y_{0}\Vert\}\leq r_{0}$. Since $y$ solves  \eqref{Filippov1}, Proposition \ref{Main_Result_Red} (applied with $y_{0}$ in place of $x_{0}$) ensures that $y(\cdot)$
\begin{enumerate}
    \item [i)] also is a solution of  \eqref{Reduced}, with $y_{0}$ as initial condition, and satisfies $y(t)\in C(t)$ for all $t\in I$.

    \item [ii)] satisfies
\begin{align*}
    \Vert \dot{y}(t)\Vert \leq 2 \left(\gamma(t) + c(t)\right) (\theta(t, h) + \omega(t, g)) + \varepsilon(t, h) + 2 g(t) \quad \textrm{ a.e. } t\in I,
\end{align*}
and $\Vert y(t)\Vert \leq \theta(t, h) + \omega(t, g)= R(t, h, g)$ for all $t\in I$,  where  $\gamma(\cdot)$, $\theta(\cdot, h)$ and $\omega(\cdot, g)$ defined as in \eqref{eta-eq}.
\end{enumerate}
Since $y$ is a solution of \eqref{Reduced}, Lemma~\ref{lem:meas-subdiff} and Proposition \ref{meassum} ensure the existence of  measurable selections  $d_{y}$ and $b$ of the set-valued mappings $t\rightrightarrows \partial d_{C(t)}(y(t))$ and $t\rightrightarrows \mathbb{B}$
such that
\begin{align*}
f(t): = \dot{y}(t) - f_{1}\left(t,y\left(t\right)\right) - \int_{0}^t f_2(t,s,y(s))\, ds + (L_C + \nu(t, h, g)) d_{y}(t) - z(t)\\
\in F(t, y(t) + h(t))\quad \textrm{ for a.e. } t\in I,
\end{align*}
where $z(t)= g(t)b(t)$, and $\gamma(\cdot)$, $\nu(\cdot, h, g)$, $\theta(\cdot, h)$, $\omega(\cdot, g)$ are defined as above.

\begin{claim}{1}
There exists a  solution $x(\cdot)$ to \eqref{Sweeping-Dif1} with $x(0)=x_{0}$.
\end{claim}
\begin{claimproof}{1}
Set $\overline{R} := R(T, h, g) + \left\|h\right\|_{\infty}$. Then, by the assumption \ref{HOSL}, there exists a nonnegative integrable function $L_{F}^{\overline{R}}\colon I\to \mathbb{R}_{+}$ (see Remark~\ref{rem:OSL-normalization}) such that, for a.e. $t\in I$, for all $u,w\in \overline{R}\mathbb{B}$ and every $\xi\in F(t,u)$, there exists $\zeta \in F(t,w)$ satisfying
\begin{align*}
    \left\langle u-w, \xi - \zeta\right\rangle\leq L_{F}^{\overline{R}}(t)\left\|u-w \right\|^{2}.
\end{align*}
Write $p_{t}(x):=\operatorname{proj}_{R(t, 0, 0)\mathbb{B}}(x)$. The set-valued mapping $S\colon I\times \mathcal{H}\rightrightarrows \mathcal{H}$ is defined as the set of all $v\in F\bigl(t,p_{t}(x)\bigr)$ satisfying the following inequality
\begin{align*}
   \left\langle (y(t) + h(t)) - p_{t}(x), f(t) - v\right\rangle\leq L_{F}^{\overline{R}}(t)\left\|(y(t) + h(t)) - p_{t}(x) \right\|^{2}.
\end{align*}
The following properties hold for the multifunction $S$:
\begin{enumerate}[label=\textnormal{(\roman*)}]
    \item $S(t,x)$ has nonempty, convex and closed values.
    \item The map $t\rightrightarrows \gph S(t,\cdot)$ is measurable.
    \item For a.e. $t\in I$, $\gph S(t,\cdot)$ is closed in $\mathcal{H}\times \mathcal{H}_w$.
    \item For all $x\in \H$ and a.e. $t\in I$,
        \begin{equation*}
            \begin{aligned}
                \Vert S(t,x)\Vert:=\sup\{\Vert w\Vert \colon w\in S(t,x)\}\leq c(t) \Vert x\Vert +m(t).
            \end{aligned}
        \end{equation*}
    \item $S$ inherits the compactness assumption \ref{H4F} from $F$.
\end{enumerate}

\noindent Indeed, note that $y(t) + h(t)$, $p_{t}(x)\in \overline{R}\mathbb{B}$ since
\begin{align*}
    \Vert y(t) + h(t) \Vert\leq \Vert y(t) \Vert + \Vert h \Vert_{\infty}\leq R(t, h, g) + \Vert h \Vert_{\infty}\leq R(T, h, g) + \Vert h \Vert_{\infty} = \overline{R}.
\end{align*}
Now, due to \eqref{eq:R-monotone} and the monotonicity of $R(\cdot,h,g)$,
\begin{equation*}
\begin{aligned}
    R(t, 0, 0) &= \theta(t, 0) + \omega(t, 0) = \theta(t, 0)\leq \theta(t, h) + \omega(t, g)\\
    &\leq \theta(T, h) + \omega(T, g) = R(T, h, g),
\end{aligned}
\end{equation*}
so that $R(t, 0, 0)\leq R(T, h, g) + \Vert h \Vert_{\infty} = \overline{R}$, and moreover $p_{t}(x)\in R(t, 0, 0)\mathbb{B}\subset \overline{R}\mathbb{B}$.
Then, for $y(t) + h(t)$, $p_{t}(x)\in \overline{R}\mathbb{B}$ and $f(t)\in F(t, y(t) + h(t))$, assumption \ref{HOSL} guarantees that there exists $v\in F\bigl(t,p_{t}(x)\bigr)$ such that
\begin{align*}
\left\langle (y(t) + h(t)) - p_{t}(x), f(t) - v\right\rangle\leq L_{F}^{\overline{R}}(t)\left\|(y(t) + h(t)) - p_{t}(x)\right\|^{2}.
\end{align*}
Therefore, $v\in S(t, x)$ and $S(t, x)\neq \emptyset$. Moreover, $S(t,x)$ is the intersection of the closed convex set $F\bigl(t,p_{t}(x)\bigr)$ with a closed half-space, hence it is convex and closed, which shows (i).

Assertion (iii) follows from \ref{H2F} and the continuity of $p_{t}$: if $x_{n}\to x$ strongly and $v_{n}\rightharpoonup v$ weakly with $v_{n}\in S(t,x_{n})$, then $p_{t}(x_{n})\to p_{t}(x)$ strongly, so that $v\in F\bigl(t,p_{t}(x)\bigr)$ by Remark~\ref{rem:HF}(iii), while the defining inequality passes to the limit because the pairing of a strongly convergent sequence with a weakly convergent one converges.

As for (ii), consider the map $\Phi\colon I\times \H\times\H\to \mathbb{R}$ defined by
\begin{align*}
\Phi(t,x,v):=\left\langle (y(t)+h(t))-p_{t}(x),\, f(t)-v\right\rangle-L_{F}^{\overline{R}}(t)\left\Vert (y(t)+h(t))-p_{t}(x)\right\Vert^{2},
\end{align*}
so that $\gph S=\left\{(t,x,v)\colon \bigl(t,p_{t}(x),v\bigr)\in \gph F\right\}\cap\left\{(t,x,v)\colon \Phi(t,x,v)\leq 0\right\}$, where $\gph F$ is as in Remark~\ref{rem:HF}(iv). The map $(r,x)\mapsto \operatorname{proj}_{r\mathbb{B}}(x)$ is continuous on $\mathbb{R}_{+}\times \H$ and $t\mapsto R(t,0,0)$ is continuous, so $(t,x)\mapsto p_{t}(x)$ is a Carath\'eodory map and therefore $\mathcal{I}\otimes\mathcal{B}(\H)$-measurable. Consequently $(t,x,v)\mapsto (t,p_{t}(x),v)$ is measurable and the first set above is $\mathcal{I}\otimes\mathcal{B}(\H\times\H)$-measurable by Remark~\ref{rem:HF}(iv). Since $y$, $h$, $f$ and $L_{F}^{\overline{R}}$ are measurable and $\Phi(t,\cdot,\cdot)$ is continuous, $\Phi$ is $\mathcal{I}\otimes\mathcal{B}(\H\times \H)$-measurable and the second set is measurable as well. Hence $\gph S$ is $\mathcal{I}\otimes \mathcal{B}(\H\times\H)$-measurable and, by (i) and (iii), $\gph S(t,\cdot)$ is closed in $\H\times \H$ for a.e. $t\in I$. Since $\mathcal{I}$ is complete, the projection theorem \cite[Theorem~III.30]{Castaing_Valadier-1977} yields that $t\rightrightarrows \gph S(t,\cdot)$ is measurable, which is (ii).

As for (v), let $r>0$ and $A\subset r\mathbb{B}$ be nonempty. Since $p_{t}$ is $1$-Lipschitz and $\Vert p_{t}(x)\Vert \leq \Vert x\Vert$, the set $A':=p_{t}(A)$ satisfies $A'\subset r\mathbb{B}$ and $\chi(A')\leq \chi(A)$, whence, using $S(t,A)\subset F(t,A')$,
\begin{align*}
\chi(S(t,A))\leq \chi\left(F(t,A')\right)\leq k_{r}(t)\chi(A')\leq k_{r}(t)\chi(A),
\end{align*}
so that $S$ satisfies \ref{H4F} with the same functions $k_{r}$. Due to \ref{H3F} and the definition of $S(t, x)$, we have
\begin{equation*}
\begin{aligned}
\Vert S(t,x)\Vert \leq \Vert F\bigl(t,p_{t}(x)\bigr)\Vert\leq c(t) \Vert p_{t}(x)\Vert + m(t)\leq c(t) \Vert x\Vert + m(t),
\end{aligned}
\end{equation*}
which proves (iv). Then, by \cite[Theorem~7]{JNV20252}, under either \ref{HC3} or \ref{H4F}, there exists an absolutely continuous solution $x\colon I\to \H$ satisfying, for $x_{0}\in \mathcal{K}$, the differential inclusion
{\small 
\begin{equation}\label{Problema_S}
\left\{
\begin{aligned}
\dot{x}(t) & \in -N_{C(t)}(x(t))+f_1(t,x(t))+\int_0^t f_2(t,s,x(s))\, ds + S(t,x(t)) & \textrm{ a.e. } t\in I,\\
x(0)&=x_0.
\end{aligned}
\right.
\end{equation}}
By (i)-(iv), the set-valued map $S$ satisfies \ref{HF}, so that Proposition~\ref{Main_Result_Red}, applied to \eqref{Problema_S} with $S$ in place of $F$, $h\equiv 0$ and $g\equiv 0$ (recall that $\Vert x_{0}\Vert \leq r_{0}$, and that by (iv) the map $S$ obeys \ref{H3F} with the same functions $c$ and $m$ as $F$, so that the associated bounds are $R(\cdot,0,0)$ and $L_{C}+\nu(\cdot,0,0)$), yields that $x(t)\in C(t)$ for all $t\in I$, that $x(\cdot)$ solves the reduced inclusion, and that
\begin{align*}
    \Vert \dot{x}(t)\Vert & \leq L_{C}+\nu(t, 0, 0)\quad \textrm{ a.e. } t\in I,\\
    \Vert x(t)\Vert & \leq R(t, 0, 0)\quad \textrm{ for all } t\in I,
\end{align*}
which implies that $p_{t}(x(t))=x(t)$ and hence $S(t, x(t))\subset F(t, x(t))$ for a.e. $t\in I$. Thus, $x(\cdot)$ is a solution of \eqref{Sweeping-Dif1}.
\end{claimproof}

\noindent Now, consider $\psi$ and $\varphi$ defined on $I$ by $\psi(t) = \left\|y(t) - x(t) \right\|$ and $\varphi(t) = \frac{1}{2}\left\|y(t) - x(t) \right\|^{2}$. Then the following facts hold:
\begin{enumerate}[label=\textnormal{(\roman*)}]
    \item The functions $\psi(\cdot)$ and $\varphi(\cdot)$ are absolutely continuous.
    \item For every $t\in I$ where $\dot{\varphi}(t)$ and $\dot{\psi}(t)$ exist, $\dot{\varphi}(t) = \psi(t)\dot{\psi}(t)$.
\end{enumerate}
\begin{claim}{2}
For a.e. $t\in I$,
\begin{equation}\label{Bound_1}
\begin{aligned}
\dot{\varphi}(t)\leq &2\left(L_{F}^{\overline{R}}(t) + l_{\overline{R}}^{1}(t) +\frac{L_{C} + \nu(t, h, g)}{\rho}\right) \varphi(t)\\
  &+\sqrt{2}\left(2L_{F}^{\overline{R}}(t)\Vert h(t)\Vert + l_{\overline{R}}^{1}(t)\Vert h(t)\Vert + g(t)\right)\sqrt{\varphi(t)}\\
  &+ 2\, l_{\overline{R}}^{2}(t)\sqrt{\varphi(t)}\int_{0}^t \sqrt{\varphi(s)}\dd s + \sqrt{2}\, l_{\overline{R}}^{2}(t)\Vert h(t)\Vert\int_{0}^t \sqrt{\varphi(s)}\dd s \\
& + L_{F}^{\overline{R}}(t) \Vert h(t)\Vert^{2}+\left(g(t) + 4\left(L_{C}+\nu(t, h, g)\right)\right)\Vert h(t)\Vert.
\end{aligned}
\end{equation}
\end{claim}

\begin{claimproof}{2}
For a.e. $t\in I$, we obtain that
\begin{align*}
    \dot{\varphi}(t) & = \left\langle x(t) - y(t), \dot{x}(t) - \dot{y}(t)\right\rangle\\
     & = \left\langle \left(y(t) + h(t)\right) - x(t), \dot{y}(t) - \dot{x}(t)\right\rangle + \left\langle h(t), \dot{x}(t) - \dot{y}(t)\right\rangle.
\end{align*}

\noindent\emph{Step 1: Estimate of $\left\langle \left(y(t) + h(t)\right) - x(t), \dot{y}(t) - \dot{x}(t)\right\rangle$.}
Since $x$ is a solution of the differential inclusion \eqref{Problema_S}, Lemma~\ref{lem:meas-subdiff} and Proposition~\ref{meassum} provide measurable selections $d_x$ and $v_x$ of the set-valued mappings $t\rightrightarrows \partial d_{C(t)}(x(t))$ and $t\rightrightarrows S(t, x(t))$, respectively, such that
\begin{align*}
\dot{x}(t) = f_1(t,x(t)) + \displaystyle\int_{0}^t f_2(t,s,x(s))\, ds - (L_C+\nu(t, 0, 0))d_{x}(t) + v_{x}(t),
\end{align*}
where we used $p_{t}(x(t)) = x(t)$, which holds because $\Vert x(t)\Vert \leq R(t, 0, 0)$  $\forall t\in I$.

\noindent Writing $\lambda(t):=L_{C}+\nu(t,h,g)$ and $\lambda_{0}(t):=L_{C}+\nu(t,0,0)$, so that $0\leq \lambda_{0}(t)\leq \lambda(t)$ by \eqref{eq:R-monotone}, and subtracting the two decompositions of $\dot{y}(t)$ and $\dot{x}(t)$, we get
{\small 
\begin{equation*}
\begin{aligned}
&\left\langle \left(y(t) + h(t)\right) - x(t), \dot{y}(t) - \dot{x}(t)\right\rangle \\
&= \left\langle \left(y(t) + h(t)\right) - x(t), f(t) - v_{x}(t)\right\rangle + \langle \left(y(t) + h(t)\right) - x(t), f_1(t,y(t))-f_1(t, x(t))\rangle\\
&\quad + \left\langle  \left(y(t) + h(t)\right) - x(t),\int_0^t(f_2(t,s,y(s))-f_2(t,s,x(s)))\, ds\right\rangle\\
& \quad +\left[\lambda_{0}(t)\langle \left(y(t) + h(t)\right) - x(t), d_{x}(t)\rangle - \lambda(t)\langle \left(y(t) + h(t)\right) - x(t), d_{y}(t)\rangle\right]\\
& \quad + \langle \left(y(t) + h(t)\right) - x(t),g(t)b(t)\rangle.
\end{aligned}
\end{equation*}}
\noindent We estimate the five terms separately. We use that $x(t),y(t)\in C(t)$ and that
{\small
\begin{align*}
d_{x}(t)\in \partial d_{C(t)}(x(t))=N^{P}(C(t);x(t))\cap \mathbb{B}, \quad d_{y}(t)\in \partial d_{C(t)}(y(t))=N^{P}(C(t);y(t))\cap \mathbb{B}.
\end{align*}}
Recall also that $\Vert x(t)\Vert\leq \overline{R}$, $\Vert y(t)\Vert \leq \overline{R}$ and $\Vert y(t)+h(t)\Vert \leq \overline{R}$, and set $\Delta(t):=\Vert (y(t)+h(t))-x(t)\Vert \leq \psi(t)+\Vert h(t)\Vert$.
\begin{enumerate}[label=\textnormal{(\arabic*)}]
\item Since $v_{x}(t)\in S(t,x(t))$ and $p_{t}(x(t))=x(t)$, the definition of $S$ gives
\begin{align*}
& \left\langle \left(y(t) + h(t)\right) - x(t), f(t) - v_{x}(t)\right\rangle \leq L_{F}^{\overline{R}}(t)\Delta(t)^{2}\\
&\qquad \qquad \qquad \qquad \qquad \qquad \qquad \leq L_{F}^{\overline{R}}(t)\left(\psi^{2}(t)+2\Vert h(t)\Vert \psi(t)+\Vert h(t)\Vert^{2}\right).
\end{align*}
\item By \ref{Hf}(b) and the Cauchy-Schwarz inequality,
\begin{align*}
\langle \left(y(t) + h(t)\right) - x(t), f_1(t,y(t))-f_1(t, x(t))\rangle &\leq \Delta(t)\, l_{\overline{R}}^{1}(t)\psi(t)\\
&\leq l_{\overline{R}}^{1}(t)\left(\psi^{2}(t)+\Vert h(t)\Vert \psi(t)\right).
\end{align*}
\item By \ref{Hg}(b) and the Cauchy-Schwarz inequality,
\begin{align*}
&\langle  \left(y(t) + h(t)\right) - x(t),\int_0^t(f_2(t,s,y(s))-f_2(t,s,x(s)))\, ds \rangle\\
&\qquad \qquad \qquad \qquad \qquad \qquad\qquad \leq  l_{\overline{R}}^{2}(t)\left(\int_0^t \psi(s)\, ds\right) \left(\psi(t)+\Vert h(t)\Vert\right).
\end{align*}
\item For the normal terms, we write
\begin{align*}
&\lambda_{0}(t)\langle \left(y(t) + h(t)\right) - x(t), d_{x}(t)\rangle - \lambda(t)\langle \left(y(t) + h(t)\right) - x(t), d_{y}(t)\rangle\\
&\qquad\qquad\qquad\qquad =\lambda_{0}(t)\left\langle \left(y(t) + h(t)\right) - x(t), d_{x}(t)-d_{y}(t)\right\rangle\\
&\qquad\qquad\qquad\qquad\quad -\left(\lambda(t)-\lambda_{0}(t)\right)\left\langle \left(y(t) + h(t)\right) - x(t), d_{y}(t)\right\rangle.
\end{align*}
On the one hand, the $\frac{1}{\rho}$-hypomonotonicity of $N^{P}(C(t);\cdot)\cap \mathbb{B}$ (Proposition~\ref{Prop2.1}(e)), applied at the pair $(x(t),y(t))\in C(t)\times C(t)$, gives $\langle d_{x}(t)-d_{y}(t),x(t)-y(t)\rangle \geq -\frac{1}{\rho}\psi^{2}(t)$, so that, since $\Vert d_{x}(t)-d_{y}(t)\Vert \leq 2$,
\begin{align*}
 \langle \left(y(t) + h(t)\right) &- x(t), d_{x}(t)-d_{y}(t) \rangle \\
& =-\left\langle x(t)-y(t),d_{x}(t)-d_{y}(t)\right\rangle+\left\langle h(t),d_{x}(t)-d_{y}(t)\right\rangle\\
&\leq \frac{1}{\rho}\psi^{2}(t)+2\Vert h(t)\Vert.
\end{align*}
On the other hand, since $d_{y}(t)\in N^{P}_{C(t)}(y(t))$, $\Vert d_{y}(t)\Vert \leq 1$ and $x(t)\in C(t)$, inequality \eqref{eq:prox-reg-def} applied to $C(t)$ yields $\langle d_{y}(t),x(t)-y(t)\rangle \leq \frac{1}{2\rho}\psi^{2}(t)$, whence
\begin{equation*}
\begin{aligned}
-\left\langle \left(y(t) + h(t)\right) - x(t), d_{y}(t)\right\rangle&=\left\langle x(t)-y(t),d_{y}(t)\right\rangle-\left\langle h(t),d_{y}(t)\right\rangle\\
&\leq \frac{1}{2\rho}\psi^{2}(t)+\Vert h(t)\Vert.
\end{aligned}
\end{equation*}
Since $0\leq \lambda_{0}(t)\leq \lambda(t)$, combining both estimates we obtain
\begin{align*}
&\lambda_{0}(t)\langle \left(y(t) + h(t)\right) - x(t), d_{x}(t)\rangle - \lambda(t)\langle \left(y(t) + h(t)\right) - x(t), d_{y}(t)\rangle\\
&\qquad \leq \lambda_{0}(t)\left(\frac{1}{\rho}\psi^{2}(t)+2\Vert h(t)\Vert\right)+\left(\lambda(t)-\lambda_{0}(t)\right)\left(\frac{1}{2\rho}\psi^{2}(t)+\Vert h(t)\Vert\right)\\
&\qquad =\frac{\lambda(t)+\lambda_{0}(t)}{2}\cdot\frac{\psi^{2}(t)}{\rho}+\left(\lambda(t)+\lambda_{0}(t)\right)\Vert h(t)\Vert \\
&\qquad \leq \frac{\lambda(t)}{\rho}\psi^{2}(t)+2\lambda(t)\Vert h(t)\Vert.
\end{align*}
\item Finally, $\langle \left(y(t) + h(t)\right) - x(t),g(t)b(t)\rangle \leq g(t)\Delta(t)\leq g(t)\left(\psi(t)+\Vert h(t)\Vert\right)$.
\end{enumerate}
Adding up (1)-(5), we get
\begin{align*}
&\left\langle \left(y(t) + h(t)\right) - x(t), \dot{y}(t) - \dot{x}(t)\right\rangle \\
&\leq\left(L_{F}^{\overline{R}}(t) + l_{\overline{R}}^{1}(t) +\frac{L_{C} + \nu(t, h, g)}{\rho}\right)\psi^{2}(t)\\
 & + \left(2L_{F}^{\overline{R}}(t)\Vert h(t)\Vert + l_{\overline{R}}^{1}(t)\Vert h(t)\Vert + l_{\overline{R}}^{2}(t)\int_{0}^t \psi(s)\, ds + g(t)\right)\psi(t)\\
& + L_{F}^{\overline{R}}(t) \Vert h(t)\Vert^{2} + \left(l_{\overline{R}}^{2}(t)\int_{0}^t \psi(s)\, ds + g(t) + 2\left(L_{C}+\nu(t,h,g)\right)\right)\Vert h(t)\Vert.
\end{align*}

\noindent\emph{Step 2: Estimate of $\left\langle h(t), \dot{x}(t) - \dot{y}(t)\right\rangle$.}
We have that
\begin{align*}
    \left\langle h(t), \dot{x}(t) - \dot{y}(t)\right\rangle &\leq \left\|h(t)\right\|\left(\left\|\dot{y}(t)\right\| + \left\|\dot{x}(t)\right\|\right)\\
 &\leq \left\|h(t)\right\|\left(\lambda(t) + \lambda_{0}(t)\right)
 \leq 2\left(L_{C}+\nu(t,h,g)\right)\Vert h(t)\Vert,
\end{align*}
where we have used \eqref{eq:R-monotone} and the velocity bounds of Proposition~\ref{Main_Result_Red}.

\noindent Therefore, from \emph{Step 1} and \emph{Step 2}, we get that
\begin{align*}
&\left\langle x(t) - y(t), \dot{x}(t) - \dot{y}(t)\right\rangle\\
&\leq\left(L_{F}^{\overline{R}}(t) + l_{\overline{R}}^{1}(t) +\frac{L_{C} + \nu(t, h, g)}{\rho}\right)\psi^{2}(t)\\
 & + \left(2L_{F}^{\overline{R}}(t)\Vert h(t)\Vert + l_{\overline{R}}^{1}(t)\Vert h(t)\Vert + l_{\overline{R}}^{2}(t)\int_{0}^t \psi(s)\, ds + g(t)\right)\psi(t)\\
& + L_{F}^{\overline{R}}(t) \Vert h(t)\Vert^{2} + \left(l_{\overline{R}}^{2}(t)\int_{0}^t \psi(s)\, ds + g(t) + 4\left(L_{C}+\nu(t,h,g)\right)\right)\Vert h(t)\Vert.
\end{align*}
Since $\dot\varphi(t)=\psi(t)\dot{\psi}(t)$ for a.e. $t\in I$, the last inequality reads
{\small 
\begin{equation}\label{Bound_2}
\begin{aligned}
\psi(t)\dot{\psi}(t)&\leq\left(L_{F}^{\overline{R}}(t) + l_{\overline{R}}^{1}(t) +\frac{L_{C} + \nu(t, h, g)}{\rho}\right)\psi^{2}(t)\\
& + \left(2L_{F}^{\overline{R}}(t)\Vert h(t)\Vert + l_{\overline{R}}^{1}(t)\Vert h(t)\Vert + l_{\overline{R}}^{2}(t)\int_{0}^t \psi(s)\, ds + g(t)\right)\psi(t)\\
& + L_{F}^{\overline{R}}(t) \Vert h(t)\Vert^{2} + \left(l_{\overline{R}}^{2}(t)\int_{0}^t \psi(s)\, ds + g(t) + 4\left(L_{C}+\nu(t,h,g)\right)\right)\Vert h(t)\Vert,
\end{aligned}
\end{equation}}
and \eqref{Bound_1} follows by substituting $\psi(t)=\sqrt{2\varphi(t)}$ and $\psi(t)\dot{\psi}(t)=\dot{\varphi}(t)$. This finishes the proof of \textit{Claim 2}.
\end{claimproof}

\noindent Finally, applying Proposition~\ref{Gronwall} to \eqref{Bound_1} with  $\beta=\varphi$, $\varepsilon=\mu$ and
\begin{alignat*}{2}
    K_1(t) & =\sqrt{2}\left(2L_{F}^{\overline{R}}(t)\Vert h(t)\Vert + l_{\overline{R}}^{1}(t)\Vert h(t)\Vert + g(t)\right),\\
    K_{2}(t) & =  2\left(L_{F}^{\overline{R}}(t) + l_{\overline{R}}^{1}(t) +\frac{L_{C} + \nu(t, h, g)}{\rho}\right),\qquad K_{3}(t)  = 2\, l_{\overline{R}}^{2}(t),\\
    K_{4}(t) & = 1, \qquad K_{5}(t)  = \sqrt{2}\, l_{\overline{R}}^{2}(t)\Vert h(t)\Vert, \qquad K_{6}(t) =  1,\\
    \mu(t) & =L_{F}^{\overline{R}}(t) \Vert h(t)\Vert^{2}+\left(g(t) + 4\left(L_{C}+\nu(t, h, g)\right)\right)\Vert h(t)\Vert,
\end{alignat*}
all of which are nonnegative and integrable on $I$ (for $L_{F}^{\overline{R}}$, see Remark~\ref{rem:OSL-normalization}(i); for $\nu(\cdot,h,g)$, see Remark~\ref{rem:theta-omega}), we obtain, for all $t\in I$,
\begin{equation*}
\begin{aligned}
       \varphi(t)&\leq \varphi(0)\times\exp\left(\int_0^t \left(\eta_{1}(s) + \eta_{2}(s)\right)\, ds\right)\\
       &+\int_0^t\left[\mu(s) + \eta_{1}(s)\right] \times\exp\left(\int_s^t \left(\eta_{1}(\tau) + \eta_{2}(\tau)\right)\, d\tau\right)\, ds,
\end{aligned}
    \end{equation*}
where, according to Proposition~\ref{Gronwall},
\begin{align*}
\eta_{1}(t) &= K_{1}(t)+K_{5}(t)\int_{0}^{t}K_{6}(s)\, ds\\
&=\sqrt{2}\left(2L_{F}^{\overline{R}}(t)\Vert h(t)\Vert + l_{\overline{R}}^{1}(t)\Vert h(t)\Vert + g(t)\right) + \sqrt{2}\, l_{\overline{R}}^{2}(t)\Vert h(t)\Vert\, t,\\
\eta_{2}(t)&= K_{2}(t)+K_{3}(t)\int_{0}^{t}K_{4}(s)\, ds\\
&=2\left(L_{F}^{\overline{R}}(t) + l_{\overline{R}}^{1}(t) +\frac{L_{C} + \nu(t, h, g)}{\rho}\right) + 2\, l_{\overline{R}}^{2}(t)\, t.
\end{align*}
Since $2\varphi(t)=\Vert x(t)-y(t)\Vert^{2}$ and $2\varphi(0)=\Vert y_{0}-x_{0}\Vert^{2}=\operatorname{dist}^{2}(y_{0},\mathcal{K})$, multiplying the last inequality by $2$ gives exactly the estimate in (a).

\noindent To prove (b), note that if $h\equiv 0$, then $\overline{R}=R(T,0,g)$ and, in inequality \eqref{Bound_2}, we get that
{\small 
\begin{align*}
\psi(t)\dot{\psi}(t)&\leq\left(L_{F}^{\overline{R}}(t) + l_{\overline{R}}^{1}(t) +\frac{L_{C} + \nu(t, 0, g)}{\rho}\right)\psi^{2}(t) + \left(g(t) + l_{\overline{R}}^{2}(t)\int_{0}^t \psi(s)\, ds \right)\psi(t)
\end{align*}}
for a.e. $t\in I$. Set $\eta_{0}(t):=L_{F}^{\overline{R}}(t) + l_{\overline{R}}^{1}(t) +\frac{L_{C} + \nu(t, 0, g)}{\rho}\geq 0$ and $b(t):=g(t) + l_{\overline{R}}^{2}(t)\int_{0}^t \psi(s)\, ds\geq 0$, so that the above inequality reads $\psi(t)\dot\psi(t)\leq \psi(t)\bigl(\eta_{0}(t)\psi(t)+b(t)\bigr)$. Let $t\in\, ]0,T[$ be a point of differentiability of $\psi$ at which the inequality holds. If $\psi(t)>0$, dividing by $\psi(t)$ gives $\dot\psi(t)\leq \eta_{0}(t)\psi(t)+b(t)$. If $\psi(t)=0$, then $t$ is a global minimum of the nonnegative function $\psi$, so that $\dot\psi(t)=0\leq b(t)=\eta_{0}(t)\psi(t)+b(t)$. In both cases (see also \cite[Lemma~2.3]{Vilches2024}),
\begin{align*}
\dot{\psi}(t)\leq g(t) + \left(L_{F}^{\overline{R}}(t) + l_{\overline{R}}^{1}(t) +\frac{L_{C} + \nu(t, 0, g)}{\rho}\right)\psi(t) + l_{\overline{R}}^{2}(t)\int_{0}^t \psi(s)\, ds
\end{align*}
for a.e. $t\in I$. 
\noindent Integrating this on $[0,t]$ and arguing as in the proof of Lemma~\ref{lem:volterra-gronwall}, i.e., passing to the nondecreasing envelope $\Psi(t):=\max_{s\in [0,t]}\psi(s)$ and using that $\int_0^s \psi(\tau)\, d\tau \leq s\, \Psi(s)$, we obtain
\begin{align*}
\Psi(t)\leq \psi(0)+\int_0^t g(s)\, ds+\int_0^t \eta(s)\Psi(s)\, ds \quad \textrm{ for all } t\in I,
\end{align*}
and the classical Gr\"onwall inequality yields (see also \cite[Theorem~3.2]{Vilches2024})
{\small\begin{align*}
       \Vert x(t)-y(t)\Vert\leq \left\|y(0) - x(0)\right\|\times\exp\left(\int_0^t \eta(s)\, ds\right) + \int_0^t g(s)\exp\left(\int_s^t \eta(\tau)\, d\tau\right)\, ds,
    \end{align*}}
where $ \eta(t):= L_{F}^{\overline{R}}(t) + l_{\overline{R}}^{1}(t) +\frac{L_{C} + \nu(t, 0, g)}{\rho} + l_{\overline{R}}^{2}(t)\, t$.\\
\noindent Furthermore, since $x_{0}\in \mathcal{K}$ satisfies $\left\|y(0) - x(0)\right\| = \operatorname{dist}(y_0, \mathcal{K})$, we can write this inequality as it appears in (b), which proves the theorem.
\end{proof}

\section{Applications}\label{section-appl}

\subsection{An Application to Optimal Control}\label{A_Opt_Cont}

\noindent Let $\mathcal{K}\subset C(0)$ be a nonempty bounded set. We define the \emph{attainable set} of the differential inclusion \eqref{Sweeping-Dif1} at time $t\in I$ as
\begin{align*}
			A\left(t, \mathcal{K}\right) := \{ x(t): x(\cdot) \textrm{ solves } \eqref{Sweeping-Dif1} \textrm{ with }  x(0)\in \mathcal{K}\}.
			\end{align*}
We impose the condition  $\mathcal{K}\subset C(0)$ to ensure that the above set is nonempty. Indeed, under the assumptions of Theorem~\ref{Filippov_T},  \cite[Theorem~7]{JNV20252} provides, for every $x_{0}\in \mathcal{K}\subset C(0)$, at least one absolutely continuous solution of \eqref{Sweeping-Dif1}. Moreover, by Proposition~\ref{Main_Result_Red} (applied with $h\equiv 0$ and $g\equiv 0$), every such solution satisfies $x(t)\in C(t)$ and $\Vert x(t)\Vert \leq R(t,0,0)$ for all $t\in I$, so that $A(t,\mathcal{K})$ is a nonempty bounded subset of $C(t)$ and the Hausdorff distance between attainable sets is finite.  No compactness of $A(t,\mathcal{K})$ is needed in what follows: only its nonemptiness and boundedness are used. Recall that, for nonempty bounded sets $A,B\subset \H$,
\begin{align*}
\operatorname{Haus}(A,B):=\max\left\{\sup_{a\in A}d_{B}(a),\ \sup_{b\in B}d_{A}(b)\right\}.
\end{align*}
The following result, a consequence of Theorem \ref{Filippov_T}, demonstrates the continuity of the attainable set with respect to the Hausdorff distance.
\begin{corollary}\label{cor:attainable}
Suppose that the assumptions of Theorem~\ref{Filippov_T} are satisfied with $h\equiv 0$ and $g\equiv 0$. Let $\mathcal{K}_{1}$, $\mathcal{K}_{2}\subset C(0)$ be nonempty compact subsets, and let the constant $r_{0}$ of \eqref{eta-eq} be given by
\begin{align}\label{eq:r0-cor}
r_{0}:=\max\left\{\Vert z\Vert \colon z\in \mathcal{K}_{1}\cup \mathcal{K}_{2}\right\}.
\end{align}
Then, the following inequality holds:
\begin{equation*}
\operatorname{Haus}\left(A\left(t,  \mathcal{K}_{1}\right),A\left(t,  \mathcal{K}_{2}\right)\right)\leq \exp\left(\displaystyle\int_0^t \eta(s)\, ds\right)\operatorname{Haus}\left(\mathcal{K}_{1}, \mathcal{K}_{2}\right)\quad \textrm{ for all } t\in I,
\end{equation*}
where $\overline{R}:=R(T,0,0)$ and $ \eta(t):= L_{F}^{\overline{R}}(t) + l_{\overline{R}}^{1}(t) +\frac{L_{C} + \nu(t, 0, 0)}{\rho} + l_{\overline{R}}^{2}(t)\, t$.
\end{corollary}
\begin{proof} Since $y_{0}\in \mathcal{K}_{i}$ and $x_{0}\in \mathcal{K}_{j}$ for the trajectories considered below, the choice \eqref{eq:r0-cor} guarantees that hypothesis \eqref{eq:r0} of Theorem~\ref{Filippov_T} is satisfied for all $i,j\in \{1,2\}$. Note also that $\overline{R}$, and hence $\eta$, depends on $\mathcal{K}_{1}$ and $\mathcal{K}_{2}$ only through $r_{0}$. In particular, $\eta$ does not depend on the initial condition $y_{0}$ and is symmetric in $\mathcal{K}_{1},\mathcal{K}_{2}$. Fix $t\in I$. We first show that
\begin{align*}
A\left(t,  \mathcal{K}_{1}\right)\subset A\left(t,  \mathcal{K}_{2}\right) + \exp\left(\displaystyle\int_0^t \eta(s)\, ds\right)\operatorname{Haus}\left(\mathcal{K}_{1}, \mathcal{K}_{2}\right)\mathbb{B}.
\end{align*}
Indeed, let $\overline{y}$ $\in A\left(t,\mathcal{K}_{1}\right)$. By definition, there exists a trajectory $y\colon I \to \mathcal{H}$ solving \eqref{Sweeping-Dif1} with $y_{0}:=y(0)\in \mathcal{K}_{1}$ and $y(t)=\overline{y}$. Since $\mathcal{K}_{1}\subset C(0)$, the function $y(\cdot)$ is a solution of \eqref{Filippov1} with $h\equiv 0$ and $g\equiv 0$. Hence, applying assertion (b) of Theorem~\textnormal{\ref{Filippov_T}} with the compact set $\mathcal{K}_{2}$, and observing that the integral term vanishes because $g\equiv 0$, there exists a solution $x(\cdot)$ of \eqref{Sweeping-Dif1} with $x(0)\in \operatorname{Proj}_{\mathcal{K}_{2}}(y_{0})\subset \mathcal{K}_{2}$ satisfying
\begin{equation*}
\begin{aligned}
\Vert x(s)-y(s)\Vert & \leq \exp\left(\int_0^{s} \eta(\tau)\, d\tau\right)\operatorname{dist}(y_{0},\mathcal{K}_{2})
& \textrm{ for all }  s\in I.
\end{aligned}
\end{equation*}
In particular, $x(t)\in A(t,\mathcal{K}_{2})$ and, since $y_{0}\in \mathcal{K}_{1}$,
\begin{equation*}
\operatorname{dist}(y_{0},\mathcal{K}_{2})\leq \sup_{z\in \mathcal{K}_{1}}\operatorname{dist}(z,\mathcal{K}_{2})\leq \operatorname{Haus}(\mathcal{K}_{1},\mathcal{K}_{2}),
\end{equation*}
so that $\overline{y}=y(t)\in x(t)+\exp\left(\int_0^t \eta(s)\, ds\right)\operatorname{Haus}(\mathcal{K}_{1},\mathcal{K}_{2})\mathbb{B}$, which proves the above inclusion. Exchanging the roles of $\mathcal{K}_{1}$ and $\mathcal{K}_{2}$, which is licit because $r_{0}$ and $\eta$ are symmetric in $\mathcal{K}_{1},\mathcal{K}_{2}$, we obtain
\begin{equation*}
A\left(t,  \mathcal{K}_{2}\right)\subset A\left(t,  \mathcal{K}_{1}\right) + \exp\left(\displaystyle\int_0^t \eta(s)\, ds\right)\operatorname{Haus}\left(\mathcal{K}_{1}, \mathcal{K}_{2}\right)\mathbb{B},
\end{equation*}
and both inclusions together give the desired inequality.
\end{proof}
Since the function $\eta$ depends on $\mathcal{K}_{1}$ and $\mathcal{K}_{2}$ only through the constant $r_{0}$ of \eqref{eq:r0-cor}, Corollary~\ref{cor:attainable} shows that, for each $t\in I$, the map $\mathcal{K}\mapsto A(t,\mathcal{K})$ is Lipschitz continuous with respect to the Hausdorff distance, uniformly on families of compact subsets of $C(0)$ contained in a fixed bounded set. Indeed, if $\mathcal{K}_{1},\mathcal{K}_{2}\subset C(0)\cap r\mathbb{B}$ for a fixed $r>0$, then $r_{0}\leq r$ in \eqref{eq:r0-cor}, hence $R(T,0,0)$, and therefore $\overline{R}$ and $\eta$, may be computed with $r_{0}=r$. We observe that the resulting Lipschitz constant $\exp\bigl(\int_{0}^{T}\eta\bigr)$ is uniform over all such pairs.  In particular, $A(t,\cdot)$ is continuous at every such $\mathcal{K}$.

Assertion (a) of Theorem~\ref{Filippov_T} provides, in the same way, a quantitative comparison between the attainable set of the perturbed inclusion \eqref{Filippov1} and that of \eqref{Sweeping-Dif1}.

\begin{remark}\label{rem:perturbed-attainable}
Let $g\in L^1(I;\mathbb{R}_{+})$, $h\in L^{\infty}(I;\H)$ and let $\mathcal{K}_{1},\mathcal{K}_{2}\subset C(0)$ be nonempty compact sets, with $r_{0}$ given by \eqref{eq:r0-cor}. Denote by
\begin{align*}
A_{h,g}\left(t, \mathcal{K}_{1}\right) := \{ y(t)\colon y(\cdot) \textrm{ solves } \eqref{Filippov1} \textrm{ with }  y(0)\in \mathcal{K}_{1}\}
\end{align*}
the attainable set of the perturbed inclusion, and by $e(A,B):=\sup_{a\in A}d_{B}(a)$ the excess of $A$ over $B$. Then, for all $t\in I$,
\begin{align*}
e\left(A_{h,g}(t,\mathcal{K}_{1}),A(t,\mathcal{K}_{2})\right)^{2}&\leq \operatorname{Haus}^{2}(\mathcal{K}_{1},\mathcal{K}_{2})\exp\left(\int_0^t \left(\eta_{1}+\eta_{2}\right)\right)\\
&\quad +2\int_0^t \left[\mu(s)+\eta_{1}(s)\right]\exp\left(\int_s^t \left(\eta_{1}+\eta_{2}\right)\right)\, ds,
\end{align*}
with $\mu$, $\eta_{1}$ and $\eta_{2}$ as in Theorem~\ref{Filippov_T}. Indeed, it suffices to apply assertion (a) to each solution $y(\cdot)$ of \eqref{Filippov1} issued from $y_{0}\in \mathcal{K}_{1}$ and to bound $\operatorname{dist}(y_{0},\mathcal{K}_{2})$ by $\operatorname{Haus}(\mathcal{K}_{1},\mathcal{K}_{2})$ as above (the estimate holds vacuously if $A_{h,g}(t,\mathcal{K}_{1})=\emptyset$). Only the excess, and not the Hausdorff distance, can be estimated in this way: Theorem~\ref{Filippov_T} produces a solution of \eqref{Sweeping-Dif1} close to a given solution of \eqref{Filippov1}, and not conversely. Taking $\mathcal{K}_{1}=\mathcal{K}_{2}=\mathcal{K}$ and letting $\Vert h\Vert_{\infty}\to 0$ and $\Vert g\Vert_{L^1}\to 0$, the right-hand side tends to $0$, which quantifies the stability of the attainable set with respect to the perturbations $h$ and $g$. To see this, fix $\Vert h\Vert_{\infty}\leq 1$ and $\Vert g\Vert_{L^{1}}\leq 1$ and let $\overline{R}_{\star}$ be the value of $\overline{R}$ associated with these bounds. By Remark~\ref{rem:monotone-radii} the local constants $L_{F}^{\overline{R}}$, $l^{1}_{\overline{R}}$ and $l^{2}_{\overline{R}}$ are then dominated by $L_{F}^{\overline{R}_{\star}}$, $l^{1}_{\overline{R}_{\star}}$ and $l^{2}_{\overline{R}_{\star}}$, while $\int_{0}^{T}\bigl(L_{C}+\nu(s,h,g)\bigr)\, ds=R(T,h,g)-r_{0}$ remains bounded by Remark~\ref{rem:theta-omega}. Consequently $\int_{0}^{T}(\eta_{1}+\eta_{2})$ stays bounded, whereas $\int_{0}^{T}(\mu+\eta_{1})\to 0$, since every term of $\mu+\eta_{1}$ carries a factor $\Vert h(t)\Vert$ or $g(t)$.
\end{remark}

\subsection{A fishery model with a one-sided Lipschitz harvesting rule that is not Lipschitz}\label{subsec:fishery}

We close this section with a model in which the multivalued perturbation is one-sided Lipschitz with constant zero, while it fails to be Lipschitz continuous with respect to the Hausdorff distance on every ball of $\H$. A Filippov-type theorem requiring the Hausdorff Lipschitz continuity of $F(t,\cdot)$ is therefore not applicable, whereas Theorem~\ref{Filippov_T} and Corollary~\ref{cor:attainable} are. The underlying dynamics is the spatially distributed fishery with ecological memory introduced in \cite{JNV20252}. The new ingredient is a threshold harvest-control rule of the type used in fisheries management, which is precisely what destroys Lipschitz continuity.

\noindent \emph{The state space and the moving set.}  Let $\Omega\subset\mathbb{R}^{d}$ be a bounded and measurable set with $0<\vert\Omega\vert<+\infty$ and set $\H:=L^{2}(\Omega)$, the value $x(\xi)$ being interpreted as the biomass density at the location $\xi\in\Omega$. Let $\mathbf{1}\in\H$ be  the constant function equal to $1$, so that $\Vert\mathbf{1}\Vert=\vert\Omega\vert^{1/2}$, and let
\begin{align*}
s_{x}:=\langle x,\mathbf{1}\rangle=\int_{\Omega}x(\xi)\dd\xi
\end{align*}
denote the aggregate biomass. Given $b\colon I\to\H$ with $b(t)\geq 0$ a.e. on $\Omega$ and $\Vert b(t)-b(s)\Vert\leq L_{b}\vert t-s\vert$ for all $s,t\in I$, define the moving safe set
\begin{equation*}
C(t):=\left\{x\in\H\colon x\geq b(t)\ \textrm{ a.e. on }\Omega\right\}=b(t)+\mathcal{P},\qquad \mathcal{P}:=\left\{z\in\H\colon z\geq 0\right\}.
\end{equation*}
Each $C(t)$ is nonempty, closed and convex, hence uniformly $\rho$-prox-regular with $\rho=+\infty$, and $\operatorname{Haus}(C(t),C(s))\leq\Vert b(t)-b(s)\Vert$. Thus \ref{HC} holds with $L_{C}=L_{b}$ and $1/\rho=0$. Let us stress that \ref{HC3} fails as soon as $\dim\H=+\infty$: if $(\Omega_{n})_{n}$ are pairwise disjoint subsets of $\Omega$ of positive measure, the functions $\vert\Omega_{n}\vert^{-1/2}\mathbf{1}_{\Omega_{n}}$ belong to $\mathcal{P}\cap\mathbb{B}\subset C(t)\cap(1+\Vert b(t)\Vert)\mathbb{B}$ when $b(t)=0$ and are mutually at distance $\sqrt{2}$. Assumption \ref{H4F} is therefore the operative compactness hypothesis here, and this is what dictates the aggregate form of the harvesting rule below.

\noindent \emph{Recruitment, mortality and ecological memory.} Set $\psi(z):=z/(1+\vert z\vert)$ and let $r,\mu\colon I\times\Omega\to\mathbb{R}_{+}$ be measurable with $q(t):=\Vert r(t,\cdot)\Vert_{L^{\infty}(\Omega)}+\Vert\mu(t,\cdot)\Vert_{L^{\infty}(\Omega)}\in L^{1}(I)$. Define $f_{1}(t,x)(\xi):=r(t,\xi)\psi(x(\xi))-\mu(t,\xi)x(\xi)$, where $r$ is the saturating recruitment and $\mu$ the natural mortality. Since $\psi$ is $1$-Lipschitz and bounded by $1$, assumption \ref{Hf} holds with $l^{1}_{r}=q$ for every $r>0$ and $a_{1}=\max\{1,\vert\Omega\vert^{1/2}\}\,q$. Let $k\colon D\times\Omega\times\Omega\to\mathbb{R}$ be measurable with 
$$\Vert k(t,s,\cdot,\cdot)\Vert_{L^{2}(\Omega\times\Omega)}\leq\kappa(t) \textrm{ for a.e. } (t,s)\in D,
$$ and some $\kappa\in L^{1}(I;\mathbb{R}_{+})$, and let $f_{2}(t,s,x):=\mathcal{K}(t,s)x$ be the associated Hilbert--Schmidt operator, which models delayed recruitment and accumulated environmental effects. Then \ref{Hg} holds with $l^{2}_{r}=\kappa$ for every $r>0$ and $a_{2}(t,s)=\kappa(t)$, the latter being integrable on $D$ because $\int_{0}^{T}t\,\kappa(t)\dd t\leq T\Vert\kappa\Vert_{L^{1}(I)}$.

\noindent \emph{The harvesting rule.} Let $\underline{h},\overline{h}\colon I\to\mathbb{R}_{+}$ be measurable with $\underline{h}\leq\overline{h}$ and $\overline{h}\in L^{1}(I)$, let $U\colon I\rightrightarrows\mathbb{R}^{m}$ be measurable with nonempty compact convex values, and let $B\colon I\to\mathcal{L}(\mathbb{R}^{m},\H)$ be strongly measurable with 
$$d_{U}(t):=\Vert B(t)\Vert\max_{u\in U(t)}\vert u\vert\in L^{1}(I).
$$ Finally, let $\gamma\colon\mathbb{R}\rightrightarrows\mathbb{R}$ be a maximal monotone operator with $\operatorname{dom}\gamma=\mathbb{R}$ and $\gamma(\sigma)\subset[0,\overline{\zeta}\,]$ for all $\sigma$, and write $\gamma(\sigma)=[\gamma^{-}(\sigma),\gamma^{+}(\sigma)]$. The multivalued perturbation is
\begin{equation}\label{eq:fishery-F}
F(t,x):=\Bigl\{-h\,x+B(t)u-\zeta\,\mathbf{1}\ \colon\ h\in[\underline{h}(t),\overline{h}(t)],\ u\in U(t),\ \zeta\in\gamma(s_{x})\Bigr\}.
\end{equation}
Here $h$ is an admissible fishing-mortality rate, $B(t)u$ describes a finite dimensional family of spatial management interventions, and the last term is a harvest-control rule triggered by the \emph{aggregate} biomass: the prototype is the hockey-stick rule $\gamma=\overline{\zeta}\,\partial(\cdot-\sigma_{0})^{+}$, that is, $\gamma(\sigma)=\{0\}$ for $\sigma<\sigma_{0}$, $\gamma(\sigma_{0})=[0,\overline{\zeta}\,]$ and $\gamma(\sigma)=\{\overline{\zeta}\,\}$ for $\sigma>\sigma_{0}$, which prescribes no additional extraction below the reference level $\sigma_{0}$ and a constant extraction above it. The rule acts along the fixed direction $\mathbf{1}$. A pointwise-in-$\xi$ threshold would generate a term with no compactness and would destroy \ref{H4F}, which, as observed above, cannot be dispensed with in this model.

\begin{proposition}\label{prop:fishery}
Under the above hypotheses, the set-valued map $F$ defined in \eqref{eq:fishery-F} satisfies \ref{HF}, \ref{H1F}, \ref{H2F}, \ref{H3F} with $c=\overline{h}$ and $m=d_{U}+\overline{\zeta}\vert\Omega\vert^{1/2}$, \ref{H4F} with $k_{r}=\overline{h}$ for every $r>0$, and \ref{HOSL} with $L_{F}^{r}\equiv 0$ for every $r>0$. If, in addition, $\gamma$ is discontinuous at some $\sigma_{0}>0$, that is, $\gamma^{-}(\sigma_{0})<\gamma^{+}(\sigma_{0})$, then for every $\varrho>\sigma_{0}\vert\Omega\vert^{-1/2}$ the map $F(t,\cdot)$ is not Lipschitz continuous on $\varrho\mathbb{B}$ with respect to the Hausdorff distance.
\end{proposition}

\begin{proof}
Throughout, fix $t\in I$ and write $\Lambda(t):=[\underline{h}(t),\overline{h}(t)]$ and $\Psi(t,x):=\{-hx+B(t)u\colon h\in\Lambda(t),\ u\in U(t)\}$, so that $F(t,x)=\Psi(t,x)-\gamma(s_{x})\mathbf{1}$.

\noindent \emph{Values.} The set $F(t,x)$ is the image of the nonempty compact convex set $\Lambda(t)\times U(t)\times\gamma(s_{x})\subset\mathbb{R}\times\mathbb{R}^{m}\times\mathbb{R}$ under the linear map $(h,u,\zeta)\mapsto -hx+B(t)u-\zeta\mathbf{1}$. Hence it is nonempty, convex and compact, and \ref{HF} holds. 

\noindent \emph{Growth.} For $v\in F(t,x)$ one has $\Vert v\Vert\leq\overline{h}(t)\Vert x\Vert+d_{U}(t)+\overline{\zeta}\vert\Omega\vert^{1/2}$, which is \ref{H3F} with the announced $c$ and $m$, both integrable on $I$.

\noindent \emph{Closedness of the graph.} Let $x_{n}\to x$ strongly and $v_{n}\rightharpoonup v$ weakly with $v_{n}\in F(t,x_{n})$, say $v_{n}=-h_{n}x_{n}+B(t)u_{n}-\zeta_{n}\mathbf{1}$ with $h_{n}\in\Lambda(t)$, $u_{n}\in U(t)$ and $\zeta_{n}\in\gamma(s_{x_{n}})\subset[0,\overline{\zeta}\,]$. All three sequences range in compact sets, so, along a subsequence, $h_{n}\to h\in\Lambda(t)$, $u_{n}\to u\in U(t)$ and $\zeta_{n}\to\zeta$. Since $s_{x_{n}}\to s_{x}$ and $\gph\gamma$ is closed, $\zeta\in\gamma(s_{x})$. Consequently $v_{n}\to-hx+B(t)u-\zeta\mathbf{1}$ strongly, and this limit equals $v$. Thus $v\in F(t,x)$ and \ref{H2F} holds.

\noindent \emph{Measurability.} The map $t\rightrightarrows\Lambda(t)$ is measurable, and so is $t\rightrightarrows B(t)U(t)$: if $U(t)=\overline{\{u_{n}(t)\colon n\in\mathbb{N}\}}$ is a Castaing representation, then 
$$
B(t)U(t)=\overline{\{B(t)u_{n}(t)\colon n\in\mathbb{N}\}}
$$
and each $t\mapsto B(t)u_{n}(t)$ is measurable. Hence $t\rightrightarrows\Psi(t,x)$ is measurable for each fixed $x$, while $\Psi(t,\cdot)$ is $\overline{h}(t)$-Lipschitz for the Hausdorff distance.    Therefore $(t,x,w)\mapsto d(w,\Psi(t,x))$ is measurable in $t$ and continuous in $(x,w)$, hence $\mathcal{I}\otimes\mathcal{B}(\H\times\H)$-measurable. The functions $\gamma^{\pm}$ are nondecreasing, hence Borel, and, by continuity in $\lambda$,
\begin{align*}
v\in F(t,x)\iff\inf_{\lambda\in\mathbb{Q}\cap[0,1]}d\Bigl(v+\bigl((1-\lambda)\gamma^{-}(s_{x})+\lambda\gamma^{+}(s_{x})\bigr)\mathbf{1},\ \Psi(t,x)\Bigr)=0.
\end{align*}
Thus $\gph F$ is $\mathcal{I}\otimes\mathcal{B}(\H\times\H)$-measurable and $\gph F(t,\cdot)$ is closed in $\H\times\H$.  Since $\mathcal{I}$ is complete and $\H\times\H$ is Polish, \cite[Theorem~III.30]{Castaing_Valadier-1977} yields \ref{H1F}.

\noindent \emph{Compactness.} Let $r>0$ and $\emptyset\neq A\subset r\mathbb{B}$. Since $U(t)$ is compact and $B(t)$ is bounded and linear, $B(t)U(t)$ is compact.  The segment $[0,\overline{\zeta}\,]\mathbf{1}$ is compact as well, so both have vanishing measure of noncompactness. Moreover $\chi(\Lambda(t)A)\leq\overline{h}(t)\chi(A)$: given $\delta>0$, cover the compact set $\Lambda(t)$ by finitely many points $h_{1},\ldots,h_{N}$ within $\delta$, so that $\Lambda(t)A\subset\bigcup_{i}h_{i}A+\delta r\mathbb{B}$, and use that $\chi$ is invariant under finite unions through the maximum, that $\chi(h_{i}A)=\vert h_{i}\vert\chi(A)$, and let $\delta\downarrow 0$. By subadditivity, $\chi(F(t,A))\leq\overline{h}(t)\chi(A)$, which is \ref{H4F} with $k_{r}=\overline{h}$, independently of $r$.

\noindent \emph{One-sided Lipschitz property.} Let $x,y\in\H$ and $v=-h_{\ast}x+B(t)u_{\ast}-\zeta_{x}\mathbf{1}\in F(t,x)$, with $h_{\ast}\in\Lambda(t)$, $u_{\ast}\in U(t)$ and $\zeta_{x}\in\gamma(s_{x})$. Pick any $\zeta_{y}\in\gamma(s_{y})$, which is possible because $\operatorname{dom}\gamma=\mathbb{R}$, and set $w:=-h_{\ast}y+B(t)u_{\ast}-\zeta_{y}\mathbf{1}\in F(t,y)$. Since $\langle x-y,\mathbf{1}\rangle=s_{x}-s_{y}$, the monotonicity of $\gamma$ gives
\begin{align*}
\langle x-y,v-w\rangle=-h_{\ast}\Vert x-y\Vert^{2}-(\zeta_{x}-\zeta_{y})(s_{x}-s_{y})\leq-\underline{h}(t)\Vert x-y\Vert^{2}\leq 0,
\end{align*}
so that \ref{HOSL} holds with $L_{F}^{r}\equiv 0$ for every $r>0$.

\noindent \emph{Failure of Lipschitz continuity.} Recall that if $\operatorname{Haus}(A,B)\leq\delta$ for nonempty bounded sets $A,B$, then $A\subset B+\delta\mathbb{B}$ and $B\subset A+\delta\mathbb{B}$, whence $\vert\sigma_{A}(p)-\sigma_{B}(p)\vert\leq\delta$ for every $p\in\mathbb{B}$, where $\sigma_{A}(p):=\sup_{a\in A}\langle a,p\rangle$. Take $p:=\vert\Omega\vert^{-1/2}\mathbf{1}$, so that $\Vert p\Vert=1$, and set 
\begin{align*}
x_{0}:=\frac{\sigma_{0}}{\vert\Omega\vert}\mathbf{1},\qquad x_{\epsilon}:=x_{0}+\frac{\epsilon}{\vert\Omega\vert}\mathbf{1}\quad(\epsilon>0),
\end{align*}
so that $s_{x_{0}}=\sigma_{0}$, $s_{x_{\epsilon}}=\sigma_{0}+\epsilon$, $\Vert x_{0}\Vert=\sigma_{0}\vert\Omega\vert^{-1/2}$ and $\Vert x_{\epsilon}-x_{0}\Vert=\epsilon\vert\Omega\vert^{-1/2}$. Both points lie in $\varrho\mathbb{B}$ for $\epsilon$ small. Since $s_{x_{0}},s_{x_{\epsilon}}>0$, one has 
$$\sup_{h\in\Lambda(t)}\langle-hx,p\rangle=-\underline{h}(t)s_{x}\vert\Omega\vert^{-1/2} \textrm{ for  } x\in\{x_{0},x_{\epsilon}\},
$$ while $\sup_{\zeta\in\gamma(s_{x})}\langle-\zeta\mathbf{1},p\rangle=-\gamma^{-}(s_{x})\vert\Omega\vert^{1/2}$. Set $\beta(t):=\sup_{u\in U(t)}\langle B(t)u,p\rangle$. Then,  \begin{align*}
\sigma_{F(t,x_{0})}(p)-\sigma_{F(t,x_{\epsilon})}(p)&=\underline{h}(t)\frac{\epsilon}{\vert\Omega\vert^{1/2}}+\bigl(\gamma^{-}(\sigma_{0}+\epsilon)-\gamma^{-}(\sigma_{0})\bigr)\vert\Omega\vert^{1/2}\\
&\geq\bigl(\gamma^{+}(\sigma_{0})-\gamma^{-}(\sigma_{0})\bigr)\vert\Omega\vert^{1/2},
\end{align*}
because $\underline{h}\geq 0$ and, by monotonicity of $\gamma$, $\gamma^{-}(\sigma_{0}+\epsilon)\geq\gamma^{+}(\sigma_{0})$. Therefore
\begin{align*}
\operatorname{Haus}\bigl(F(t,x_{0}),F(t,x_{\epsilon})\bigr)\geq\bigl(\gamma^{+}(\sigma_{0})-\gamma^{-}(\sigma_{0})\bigr)\vert\Omega\vert^{1/2}>0,
\end{align*}
while $\Vert x_{\epsilon}-x_{0}\Vert=\frac{\epsilon}{\vert\Omega\vert^{1/2}}\xrightarrow[\epsilon\downarrow 0]{}0$.
 So no integrable function $\kappa_{\varrho}$ can satisfy the estimate of Remark~\ref{rem:OSL-normalization}(ii) on $\varrho\mathbb{B}$.
\end{proof}
Proposition~\ref{prop:fishery} shows that all the assumptions of Theorem~\ref{Filippov_T} are satisfied by the fishery model, with \ref{H4F} as the compactness hypothesis, and that the classical Lipschitz framework is genuinely out of reach. Since $L_{F}^{\overline{R}}\equiv 0$, $1/\rho=0$, $l^{1}_{\overline{R}}=q$ and $l^{2}_{\overline{R}}=\kappa$ irrespective of the radius $\overline{R}$, the exponent of Corollary~\ref{cor:attainable} reduces to $\eta(t)=q(t)+t\,\kappa(t)$, so that, for all nonempty compact sets $\mathcal{K}_{1},\mathcal{K}_{2}\subset C(0)$ and all $t\in I$,
\begin{align*}
\operatorname{Haus}\bigl(A(t,\mathcal{K}_{1}),A(t,\mathcal{K}_{2})\bigr)\leq\exp\left(\int_{0}^{t}\bigl(q(s)+s\,\kappa(s)\bigr)\dd s\right)\operatorname{Haus}(\mathcal{K}_{1},\mathcal{K}_{2}).
\end{align*}
The stability constant is thus governed exclusively by the biological data, namely the recruitment and mortality rates through $q$ and the ecological memory kernel through $\kappa$: it is independent of the admissible fishing-mortality range $[\underline{h},\overline{h}]$, of the management set $U$, of the operator $B$, and of the threshold rule $\gamma$. This is a genuine gain of the one-sided Lipschitz framework, since the sharpest Hausdorff Lipschitz modulus available for the map \eqref{eq:fishery-F} when $\gamma\equiv 0$ is $\overline{h}(t)$, and none exists at all when $\gamma$ has a jump. In the same way, assertion (a) of Theorem~\ref{Filippov_T} and Remark~\ref{rem:perturbed-attainable} quantify how far the attainable biomass profiles can drift when the aggregate biomass in the control rule is observed with an error $h$ and when the management term is subject to an outer uncertainty $g$.

\section{Conclusions}\label{Concl}
In this article, we have studied an integro-differential Volterra sweeping process with a compact outer multivalued perturbation satisfying a \textit{one-sided Lipschitz} (OSL) condition. The proof of the main result combines a reduction result establishing the equivalence between Volterra sweeping processes with outer multivalued perturbations and inner perturbations of the state variable and the corresponding unconstrained differential inclusions, existence results obtained under suitable compactness assumptions on either the moving constraint sets or the outer multivalued perturbation, an appropriate selection result, and an enhanced version of Gr\"onwall's inequality. These tools allow us to establish a Filippov-type stability theorem for this system under a \textit{one-sided Lipschitz} (OSL) condition on the outer multivalued perturbation. As an application, we prove the continuity of the associated attainable set with respect to the Hausdorff distance. We have also illustrated the scope of the one-sided Lipschitz framework on a spatially distributed fishery model with ecological memory, in which the harvest-control rule triggered by the aggregate biomass is one-sided Lipschitz with constant zero while failing to be Lipschitz continuous for the Hausdorff distance on every ball (see Proposition~\ref{prop:fishery}). For this model, the stability constant of Corollary~\ref{cor:attainable} depends only on the recruitment and mortality rates and on the memory kernel, and not on the admissible fishing-mortality range nor on the threshold rule itself. 

The results obtained in this article also open several directions for future research. One of them concerns the study of strong invariance and the minimal time problem for the nonautonomous sweeping process obtained as the particular case ($f_{1}=f_{2}=0$) of the integro-differential Volterra sweeping process considered in this article. In this regard, the works of Dontchev, R\'ios, and Wolenski \cite{Dontchev_T-Rios-Wolenski-2005}, R\'ios and Wolenski  \cite{Rios2025Wolenski}, and Hermosilla, Palladino, and Vilches \cite{hermosilla2024hamilton} seem promising for extending the present results to this nonautonomous sweeping process.

Another natural research direction concerns the study of weak stability and weak asymptotic stability for the integro-differential Volterra sweeping process considered in this article. Along these lines, the works of Roxin \cite{Roxin} on generalized dynamical systems and Smirnov \cite{Smirnov} on weak asymptotic stability provide a natural starting point for investigating these stability properties within the framework developed in this article. In particular, the continuity of attainable sets with respect to the Hausdorff distance established in this article suggests that the present framework may provide a suitable setting for the study of weak stability and weak asymptotic stability.

\section*{Acknowledgements}

This work was supported by ANID/BASAL~FB210005
(Centro de Modelamiento Matemático, CMM), and
ANID/ECOS~ECOS230027. Diana Narv\'aez was supported byANID Chile under the grant Postdoctorado~3240146.
Emilio Vilches was supported by ANID Chile under the grants  Regular~1240120 and Regular~1261728.

\section*{Data Availability}

Data sharing is not applicable to this article as no datasets were generated or analyzed during the current study.

\bibliographystyle{plain}
\bibliography{references}

\end{document}